\documentclass[11pt]{article}
\usepackage{amssymb}
\usepackage{amsmath}
\usepackage{amsfonts}
\newtheorem{precor}{{\bf Corollary}}

\newtheorem{precon}{{\bf Conjecture}}

\newtheorem{predefin}{{\bf Definition}}

\newenvironment{defin}[1]{\begin{predefin}{\hspace{-0.5
                   em}{\bf.\ }}{\rm #1}\hfill{$\blacktriangle$}}{\end{predefin}}
\newtheorem{preexm}{{\bf Example}}

\newenvironment{exm}[1]{\begin{preexm}{\hspace{-0.5
                  em}{\bf.\ }}{\rm #1}\hfill{$\blacklozenge$}}{\end{preexm}}

\newtheorem{preappl}{{\bf Application}}

\newtheorem{prelem}{{\bf Lemma}}

\newtheorem{preproof}{{\bf Proof.\ }}

\newenvironment{proof}[1]{\begin{preproof}{\rm
               #1}\hfill{$\blacksquare$}}{\end{preproof}}
\newtheorem{preclm}{{\bf Claim}}

\newtheorem{prethm}{{\bf Theorem}}

\newenvironment{thm}{\begin{prethm}{\hspace{-0.5
               em}{\bf.\ }}}{\end{prethm}}
\newtheorem{prealphthm}{{\bf Theorem}}

\newtheorem{prealphlem}{{\bf Lemma}}

\newtheorem{prepro}{{\bf Proposition}}

\newenvironment{pro}{\begin{prepro}{\hspace{-0.5
               em}{\bf.\ }}}{\end{prepro}}
\newtheorem{prequ}{{\bf Question}}

\newtheorem{prealphqu}{{\bf Question}}

\newtheorem{preprb}{{\bf Problem}}

\newcommand\bigzero{\makebox(0,0){\text{\Large0}}} 
\def\conct[#1,#2]{\mbox {${#1} \leftrightarrow {#2}$}}
\def\dconct[#1,#2]{\mbox {${#1} \rightarrow {#2}$}}
\def\deg[#1,#2]{\mbox {$d^{^{#1}}(#2)$}}
\def\mindeg[#1]{\mbox {$\delta^{^{#1}}$}}
\def\maxdeg[#1]{\mbox {$\Delta^{^{#1}}$}}
\def\outdeg[#1,#2]{\mbox {$d^{^{#1}}_{_+}(#2)$}}
\def\minoutdeg[#1]{\mbox {$\delta^{^{#1}}_{_+}$}}
\def\maxoutdeg[#1]{\mbox {$\Delta^{^{#1}}_{_+}$}}
\def\indeg[#1,#2]{\mbox {$d^{^{#1}}_{_-}(#2)$}}
\def\minindeg[#1]{\mbox {$\delta^{^{#1}}_{_-}$}}
\def\maxindeg[#1]{\mbox {$\Delta^{^{#1}}_{_-}$}}
\def\isdef{\mbox {$\ \stackrel{\rm def}{=} \ $}}
\def\dre[#1,#2,#3]{\mbox {${\cal E}^{^{#3}}(#1,#2)$}}
\def\var[#1,#2]{\mbox {${\rm Var}^{^{#1}}(#2)$}}
\def\ls[#1]{\mbox {$\xi_{_{#1}}$}}
\def\hom[#1,#2]{\mbox {${\rm Hom}({#1},{#2})$}}
\def\onvhom[#1,#2]{\mbox {${\rm Hom^{v}}(#1,#2)$}}
\def\onehom[#1,#2]{\mbox {${\rm Hom^{e}}(#1,#2)$}}
\def\core[#1]{\mbox {$#1_{_{\bullet}}$}}
\def\cay[#1,#2]{\mbox {${\rm Cay}({#1},{#2})$}}
\def\cays[#1,#2]{\mbox {${\rm Cay_{s}}({#1},{#2})$}}
\def\dirc[#1]{\mbox {$\stackrel{\rightarrow}{C}^{^{#1}}$}}
\def\cycl[#1]{\mbox {${\bf Z}^{^{#1}}$}}

\date{}
\usepackage{color}
\usepackage{blkarray}
\usepackage{calc}
\usepackage{tikz}
\usetikzlibrary{decorations.markings}
\tikzstyle{vertex}=[circle, draw, inner sep=0pt, minimum size=6pt]
\newcommand{\vertex}{\node[vertex]}
\begin{document}
\begin{center}
{\Large \bf On the Spectra of Symmetric \\ Cylindrical Constructs }\\
\vspace*{0.5cm}
{\bf  Amir Daneshgar}\\
{\it Department of Mathematical Sciences}\\
{\it Sharif University of Technology}\\
{\it P.O. Box {\rm 11155-9415}, Tehran, Iran}\\
{\tt daneshagr@sharif.ir}\\
{\bf  Ali Taherkhani}\\
{\it Department of Mathematics}\\
{\it Institute for Advanced Studies in Basic Sciences }\\
{\it P.O. Box {\rm 45195-1159}, Zanjan {\rm 45195}, Iran}\\
{\tt ali.taherkhani@iasbs.ac.ir}\\
\end{center}
\begin{abstract}
\noindent 
In this article, following [A.~Daneshgar, M.~Hejrati, M.~Madani, {\it On cylindrical graph construction and its applications}, EJC, 23(1) p1.29, 45, 2016] we study the spectra of symmetric cylindrical constructs, generalizing some well-known results on the spectra of a variety of graph products, graph subdivisions by V.~B.~Mnuhin (1980) and the spectra of GI-graphs (see [M.~Conder, T.~Pisanski, and A.~{\v{Z}}itnik, {\it GI-graphs: a new class of graphs with many symmetries}, 40, 209--231 (2014)] and references therein). In particular, we show that for bsymmetric 
cylinders with no internal vertex the spectra is actually equal to the eigenvalues of a perturbation of the base, and using this, we study the spectra of sparsifications of complete graphs by tree-cylinders. We also, show that a specific 
version of this construction gives rise to a class of highly symmetric graphs as a generalization of Petersen and Coxeter graphs. 
\end{abstract}
\section{Introduction}

Nowadays, {\it graph amalgams} is a classic chapter in graph theory while study of graph products and their spectra goes back to the early days of graph theory \cite{hakl}. The cylindrical graph construction as a generalized edge-replacement procedure, introduced in \cite{dama} (also see \cite{madani}), is a graph amalgam 
that unifies a large number of graph constructions including different kinds of products. On the other hand, study of graphs through assigning matrices that encode structural properties is also and old section 
of graph theory with a large intersection with discrete spectral geometry and its methods. These along with the main duality proved in \cite{dama} leads to the following 
basic question:\ \\

``How far the cylindrical amalgam generalizes the tensor product of matrices and its properties?" \ \\

From another point of view, looking for highly connected relatively sparse graphs is a highlight in graph theory, not only because of introducing interesting and deep problems  but also for its foundational impact in modern computer science and its applications (e.g. see \cite{avi}). In this regard, strictly speaking,  the theory of expander graphs looks
for methods of constructing highly connected sparse graphs with a relatively low spread of the spectrum. Hence, again, cylindrical construction, as a method that can be applied to sparsify highly connected dense graphs is a procedure that ought to be studied from a spectral geometric viewpoint. This approach also raises the following question: \ \\

``What can be said about the spectra of a cylindrical construct based on the spectra of the base-graph and the cylinder?"\ \\

Considering classical results in spectral geometry and recent nonconstructive results of Marcus, Spielman
and Srivastava \cite{mss} on the existence of Ramanujan graphs as iterated $2$-lifts,
provide sufficient evidence supporting  that construction of extremal highly connected sparse graphs 
are closely related to noncommutative structures where these subjects and corresponding algebras constitute a large and active area of research in modern geometry. However,
again considering $2$-lifts as cylindrical constructs gives rise to the following question:\ \\

``How far can one get close to the extremal constructions through commutative substructures?"\ \\

This article is an attempt to give some partial answers to above mentioned questions.
In Section~\ref{sec:sym}, on the one hand,  we consider a subclass of cylindrical constructs that somehow can be considered as generalized tensor products. Our basic goal in this section is to introduce a general result (Theorem~\ref{thm:main}) that describes the spectrum of the construct in terms of the spectra of its components.
On the other hand, in this regard, we introduce the concept of a {\it commutative $k$-decomposition} where in our main theorem the base-graph is decomposed into a commutative set of subgraphs by the corresponding labeling induced by the twists. 
We not only deduce a number of known results as corollaries of our main theorem, but also, in our opinion, another interesting aspect of Theorem~\ref{thm:main} is to introduces a deep connection between the theory of graph decompositions and the theory of graph spectra.

In the last section, using the concept of tree-cylinders introduced in \cite{madani},
we consider a special family of cylindrical constructions in which complete graphs are sparsified by regular tree-cylinders. In this section we analyze the spectra of such constructions in detail and we will see that such a study is closely related to the study of the spectra of regular trees in which leaves are connected to each other in a 
predefined way which is related to the twists of the construction. In this regard, we introduce a class of highly symmetric sparse graphs whose spectra is equal to the spectra of perturbed regular trees. This may also be interesting when one considers a recent result of N.~Alon \cite{alon1} (also see \cite{alon}) in which 
it is proved that there exist highly connected sparse graphs within the family of graphs constructed by adding 
random edges to a regular tree connecting leaves to their ancestors.

To add a couple of words on notations, we note that hereafter we only deal with finite simple graphs which are referred to using math-Roman font as $\mathrm{H}$, where matrices are referred to by italic-Roman  font as $H$. 
In the sequel, $\mathrm{K}_{_{n}}$ is the complete graph on $n$ vertices and $\mathrm{P}_{_{n}}$ is the path of length $n-1$.

The symbols $I$ and $J$ stand for the identity matrix and the all-one matrix, respectively, and we define 
$$\bar{I} \isdef \left[
\begin{array}{ccc}
\bigzero&&1 \\
&\reflectbox{$\ddots$}&\\
1&&\bigzero
\end{array}\right]$$
where the dimensions of these matrices are assumed to be clear from the context.
The transpose of a matrix $H$ is denoted by $H^*$ and the characteristic polynomial of a matrix $H$ 
in terms of the variable $x$ is defined as $\phi(H,x) \isdef \det(xI-H)$.

A $2 \times 2$ block matrix 
$$A = \left[
\begin{array}{cc}
A_{_{1,1}}&A_{_{1,2}} \\
A_{_{2,1}}&A_{_{2,2}}
\end{array}\right]$$
is said to be {\it bsymmetric}  if all blocks are symmetric matrices, $A_{_{1,1}}=A_{_{2,2}}$ and $A_{_{1,2}}=A_{_{2,1}}$. Clearly, the class of bsymmetric matrices of dimension $k$ constitute a subalgebra of the algebra of symmetric matrices 
which is denoted by ${\cal BS}_{_{k}}$. In this setting  ${\cal BS}$ refers to the 
whole class of bsymmetric matrices.

Moreover, let us recall a standard determinant identity for block matrices as follows;
$$\det \left[
\begin{array}{cc}
A_{_{1,1}}&A_{_{1,2}} \\
A_{_{2,1}}&A_{_{2,2}}
\end{array}\right]= \det (A_{_{1,1}}-A_{_{1,2}} A^{-1}_{_{2,2}} A_{_{2,1}}) \det (A_{_{2,2}}),$$
whenever $A_{_{2,2}}$ is an invertible matrix.

\section{Cylindrical constructs and their spectra}

Cylindrical constructs and their dual are introduced in \cite{dama} (also see \cite{madani} and references therein)
as a general edge-replacement graph construction unifying many well-known constructions 
so far and introducing a couple of new ones which deserve closer study.

In this article we will only be dealing with cylindrical constructions which are built on simple graphs using symmetric cylinders, where our major focus is on the linear algebra behind this construction. Hence, we disregard the definition of the construction 
in its most general form as it is available in \cite{dama} for the interested reader to refer to, while we provide an equivalent definition of the same concept in the special 
case appearing in this article which is in terms of the corresponding adjacency matrices
of the base-graph and the matrices assigned to the corresponding cylinders. This approach  is definitely more convenient for our linear-algebraic study and also will make our comparison to other matrix  operations as tensor products etc. easier. 

\subsection{Symmetric cylinders and their corresponding matrices}\label{sec:sym}

Since, in this article, we will only  be dealing with bsymmetric cylinders with non-intersecting bases (see \cite{dama} for the definition in a general case), in what follows a cylinder $\mathrm{H}$ will be  a simple graph,
\begin{itemize}
	\item along with two disjoint and isomorphic vertex-labeled subgraphs $$\mathrm{B}=(V(\mathrm{B})=\{b_{_{0}},b_{_{1}},\cdots,
	b_{_{\ell}}\},E(\mathrm{B}))$$ and 
	$$\mathrm{B'}=(V(\mathrm{B'})=\{b'_{_{0}},b'_{_{1}},\cdots,
	b'_{_{\ell}}\},E(\mathrm{B'}))$$
	for some integer $\ell \geq 0$, called bases,
	\item and an automorphism $\sigma \in Aut(\mathrm{H})$ such that
	       \begin{itemize}
	       	\item[{\rm a)}] We have  $\sigma(\mathrm{B})=\mathrm{B'}$ and $\sigma(\mathrm{B'})=\mathrm{B}$,
	       	\item[{\rm b)}]  For every $1 \leq i \leq \ell$ we have $\sigma(b_{_{i}})=b'_{_{i}}$ and $\sigma(b'_{_{i}})=b_{_{i}}$.
	       \end{itemize}
\end{itemize} 
Vertices that do not appear in $V(\mathrm{B} \cup \mathrm{B'})$ are called {\it inner-vertices}, and clearly, by the symmetry of the concept, at times, bases may not be distinguished in our arguments below.
Based on this definition, in what follows, we refer to a cylinder $\mathrm{H}$ as
$\mathrm{H}=(\mathrm{B},\mathrm{C})$, where $\mathrm{B}$ is the base subgraph and $\mathrm{C}$ is the subgraph induced on the inner-vertices.
In this setting, the adjacency matrix of a cylinder $\mathrm{H}=(\mathrm{B},\mathrm{C})$ is referred to as the following block matrix,
$$H=
\begin{blockarray}{cccc}
& \mathrm{B} &  \mathrm{B'}&  \mathrm{C} \\
\begin{block}{c[ccc]}
\mathrm{B}&B&E^{^{bb'}}&E^{^{bc}}\\ 
\mathrm{B'}&(E^{^{bb'}})^*&B&E^{^{b'c}}\\ 
{\mathrm{C}}&(E^{^{b{c}}})^*&(E^{^{b'{c}}})^*&C\\
\end{block}
\end{blockarray}.$$
Note that the set of inner-vertices may be empty in which case the corresponding rows and columns are excluded from the adjacency matrix. Also, in this setup, $\mathrm{B}$
and $\mathrm{B}'$ refer to the subgraphs induced on the bases, respectively, and the 
blocks of the matrix $H$ refer to the corresponding adjacency structures.

Clearly, if $\Sigma$ is the permutation matrix corresponding to the automorphism $\sigma$ partitioned as the following block matrix,
$$\Sigma=\left[
\begin{array}{ccc}
0&I&0\\ 
I&0&0\\ 
0&0&P
\end{array}\right],$$
then by the equality $\Sigma H \Sigma^* = H$ we have 
$$E^{^{{bb'}}}=(E^{^{bb'}})^*, \ E^{^{{b'c}}}=E^{^{{bc}}}P, \ E^{^{{bc}}}=E^{^{{b'c}}}P \ 
{\rm and} \ PC=CP.$$

\begin{exm}{\label{exm:bsc}{\bf A couple of symmetric cylinders}

In this example we review some very basic bsymmetric cylinders.
To start, consider the path cylinder $\mathrm{P}_{_{k+2}}$ depicted in Figure~\ref{fig:path}
whose bases are isomorphic to $\mathrm{K}_{_{1}}$ and and the subgraph induced on the inner-vertices  $\mathrm{C}$ is isomorphic to $\mathrm{P}_{_{k}}$. Clearly, a cylindrical construction 
using this bsymmetric cylinder gives rise to the well-known subdivision construction. 
Explicitly, in this case, for $k\geq 1$, we have $E^{^{{bb'}}}=0$, 
$E^{^{{bc}}}=[\begin{array}{cccc}
1&0&\cdots&0
\end{array}]_{_{1\times k}}$ 
and
$E^{^{{b'c}}}=[\begin{array}{cccc}
0&\cdots&0&1
\end{array}]_{_{1\times k}}.$

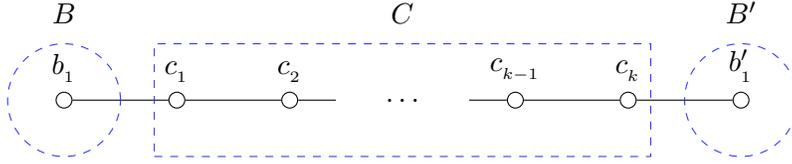
\begin{figure}[t]\label{fig:path}

 \[\begin{tikzpicture}[x=1.5cm, y=1.5cm]
 \node (b) at (1,0.5) [label=above:$B$] {};
 \node (c) at (4,0.5) [label=above:$C$] {};
 \vertex (b1) at (1,0) [label=above:$b_{_1}$] {};
 \vertex (c1) at (2,0) [label=above:$c_{_1}$] {};
 \vertex (c2) at (3,0) [label=above:$c_{_2}$] {};
 \node (c3) at (4,-0.2) [label=above:$\ldots$] {};
 \node (c32) at (3.5,0) [label=above:$$] {};
 \node (c34) at (4.5,0) [label=above:$$] {};
 \vertex (c4) at (5,0) [label=above:$c_{_{k-1}}$] {};
 \vertex (c5) at (6,0) [label=above:$c_{_{k}}$] {};
 \vertex (b'1) at (7,0) [label=above:$b'_{_1}$] {};
 \node (b) at (7,0.5) [label=above:$B'$] {};
 \path 
 (b1) edge (c1)
 (c2) edge (c1)
 (c32) edge (c2)
 (c34) edge (c4)
 (c5) edge (c4)
 (c5) edge  (b'1);
 \draw[color=blue!80,dashed](1,0) circle (0.5);
 \draw[color=blue!80, dashed] (7,0) circle (0.5);
  \draw[color=blue!80, dashed] (1.8,-0.5) -- (6.2,-0.5) -- (6.2,0.5) -- (1.8,0.5) -- (1.8,-0.5);
t\end{tikzpicture}\]
  
\caption{The path cylinder $\mathrm{P}_{_{k+2}}$(see Example~\ref{exm:subdiv}).}	
\end{figure}
\begin{figure}[h]
 \[\begin{tikzpicture}[x=1.5cm, y=1.5cm]
 \node (b) at (1,1.2) [label=above:$B_1$] {};
 \vertex (b1) at (1,0) [label=above:$$] {};
 \vertex (c1) at (2,0) [label=above:$$] {};
 \vertex (b2) at (1,1) [label=above:$$] {};
 \vertex (c2) at (2,1) [label=above:$$] {};
 \node (b) at (2,1.2) [label=above:$B'_1$] {};
 \path 
 (b1) edge (c1)
 (c2) edge (b2);
 \draw[color=blue!80, dashed] (1,0.5) ellipse (0.3 and 0.75);
 \draw[color=blue!80, dashed] (2,0.5) ellipse (0.3 and 0.75);
 \node (b) at (3,1.2) [label=above:$B_2$] {};
 \vertex (b1) at (3,0) [label=above:$$] {};
 \vertex (c1) at (4,0) [label=above:$$] {};
 \vertex (b2) at (3,1) [label=above:$$] {};
 \vertex (c2) at (4,1) [label=above:$$] {};
 \node (b) at (4,1.2) [label=above:$B'_2$] {};
 \path 
 (b2) edge (c1)
 (b1) edge (c2);
 \draw[color=blue!80, dashed](3,0.5) ellipse (0.3 and 0.75);
 \draw[color=blue!80, dashed](4,0.5) ellipse (0.3 and 0.75);
 \end{tikzpicture}\]
\caption{\label{fig:id} The identity cylinder and its twist related to random $2$-lifts (see \cite{mss}).}	
\end{figure}
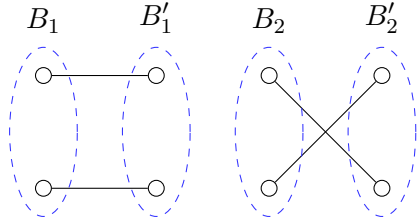

On the other hand, consider the identity cylinder and its twist depicted in Figure~\ref{fig:id}. Based on a breakthrough of Marcus, Spielman and Srivastava \cite{mss}
a random choice of these cylinders in a cylindrical construct gives rise to regular Ramanujan graphs of arbitrary degree. Note that these pair of cylinders constitute a coherent set of bsymmetric cylinders for which we have $E_{_{1}}^{^{{bb'}}}=I,E_{_{2}}^{^{{bb'}}}=\bar{I}$ (for the definition see \cite{dama,madani} or the paragraph preceding the next proposition).

As the third example, consider the $\pi$-cylinders $\sqcap$ and $\sqcup$ (see Figure~\ref{fig:pi})
that clearly form a coherent set of bsymmetric cylinders. The main point of using these cylinders is that they can be used to sparsify dense regular graphs of even degree. This as a special case of tree-cylinders 
will be discussed in Section~\ref{sec:sparsifier} in more detail (see \cite{madani}).
Note that for these cylinders we have $E_{_{1}}^{^{{bb'}}}=\left[\begin{array}{cc}
0&0\\
0&1
\end{array}\right]$ 
and
$E_{_{2}}^{^{{bb'}}}=\left[\begin{array}{cc}
1&0\\
0&0
\end{array}\right]$.
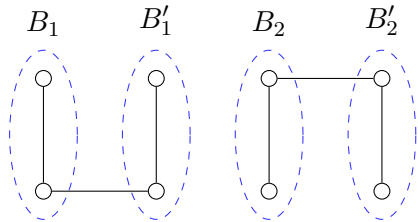
\begin{figure}[b]
 \[\begin{tikzpicture}[x=1.5cm, y=1.5cm]
 \node (b) at (1,1.2) [label=above:$B_1$] {};
 \vertex (b1) at (1,0) [label=above:$$] {};
 \vertex (c1) at (2,0) [label=above:$$] {};
 \vertex (b2) at (1,1) [label=above:$$] {};
 \vertex (c2) at (2,1) [label=above:$$] {};
 \node (b) at (2,1.2) [label=above:$B'_1$] {};
 \path 
 (b1) edge (c1)
 (c2) edge (c1)
 (b1) edge (b2);
 \draw[color=blue!80, dashed] (1,0.5) ellipse (0.3 and 0.75);
 \draw[color=blue!80, dashed] (2,0.5) ellipse (0.3 and 0.75);
 \node (b) at (3,1.2) [label=above:$B_2$] {};
 \vertex (b1) at (3,0) [label=above:$$] {};
 \vertex (c1) at (4,0) [label=above:$$] {};
 \vertex (b2) at (3,1) [label=above:$$] {};
 \vertex (c2) at (4,1) [label=above:$$] {};
 \node (b) at (4,1.2) [label=above:$B'_2$] {};
 \path 
 (b2) edge (c2)
 (c2) edge (c1)
 (b1) edge (b2);
 \draw[color=blue!80, dashed] (3,0.5) ellipse (0.3 and 0.75);
 \draw[color=blue!80, dashed] (4,0.5) ellipse (0.3 and 0.75);
 \end{tikzpicture}\]

\caption{\label{fig:pi} The $\pi$-cylinders (see \cite{dama}).}	
\end{figure}

}\end{exm}
\begin{pro}\label{pro:bs}
Given a bsymmetric cylinder $\mathrm{H}=(\mathrm{B},\mathrm{C})$, the following two statements hold and are equivalent.
\begin{itemize}
	\item[{\rm a)}] For every integer $k \geq 0$ we have 
	$$\left[
	\begin{array}{c}
	E^{^{{bc}}}\\ 
	E^{^{{b'c}}}
	\end{array}\right]
	C^k
	\left[
	\begin{array}{cc}
	(E^{^{{bc}}})^*& 
	(E^{^{{b'c}}})^*
	\end{array}\right]
	 =
	 \left[
	 	\begin{array}{cc}
	 	E^{^{{bc}}}C^k(E^{^{{bc}}})^*&E^{^{{bc}}}C^k(E^{^{{b'c}}})^* \\
	 	E^{^{{b'c}}}C^k(E^{^{{bc}}})^*&E^{^{{b'c}}}C^k(E^{^{{b'c}}})^*
	 	\end{array}\right] \in {\cal BS}.$$
	\item[{\rm b)}] For every $x \in \mathbb{R}$ which is not an eigenvalue of $C$ we have 
	$$\left[
	\begin{array}{c}
	E^{^{{bc}}}\\ 
	E^{^{{b'c}}}
	\end{array}\right]
	(xI-C)^{-1}
	\left[
	\begin{array}{cc}
	(E^{^{{bc}}})^*& 
	(E^{^{{b'c}}})^*
	\end{array}\right] \isdef
	\left[
	\begin{array}{cc}
	R_{_d}(C)&R_{_a}(C) \\
	R_{_a}(C)&R_{_d}(C)
	\end{array}\right]
	 \in {\cal BS},$$
where $R_{_{d}}(C_{_{i}})$ and $R_{_{a}}(C_{_{i}})$ are rational functions of $x$.	 
\end{itemize}	
\end{pro}
\begin{proof}{
		For Part~$(a)$, fix $k \geq 0$ and first note that using the setup preceding Example~\ref{exm:bsc}
		we have,
		$$\begin{array}{ccc}
E^{^{{b'c}}}C^k(E^{^{{b'c}}})^*&=&	E^{^{{bc}}}PC^kP^*(E^{^{{bc}}})^*\\
			&=&E^{^{{bc}}}C^k(E^{^{{bc}}})^*.
				\end{array}$$
On the other hand, 	
				$$\begin{array}{ccc}
E^{^{{bc}}}C^k(E^{^{{b'c}}})^*&=&	E^{^{{bc}}}C^kP^*(E^{^{{bc}}})^*\\
&=&E^{^{{bc}}}P^*C^k(E^{^{{bc}}})^*\\
			&=&E^{^{{b'c}}}C^k(E^{^{{bc}}})^*.
							\end{array}$$
For Part~$(b)$ note that the resolvent $(xI-C)^{-1}$ can be expressed in terms of powers of $(xI-C)$ as a polynomial by Cayley-Hamilton theorem, which itself can be expressed in terms of powers of $C$. Hence, Part~$(b)$ holds as a consequence of Part~$(a)$.
		
		Moreover, one may note that 
			$$Y(X) \isdef \left[
			\begin{array}{c}
			E^{^{{bc}}}\\ 
			E^{^{{b'c}}}
			\end{array}\right]
			X^k
			\left[
			\begin{array}{cc}
			(E^{^{{bc}}})^*& 
			(E^{^{{b'c}}})^*
			\end{array}\right]
		     $$
		is a continuous linear operator, hence, using 	Dunford-Taylor integral representation (e.g. \cite{rudin}) for the $k$th power operator one gets $(a)$.
		
}\end{proof}

It is important to note that for any bsymmetric cylinder $\mathrm{H}=(\mathrm{B},\mathrm{C})$, by Proposition~\ref{pro:bs}, essentially 
$$\left[
\begin{array}{c}
E^{^{{bc}}}\\ 
E^{^{{b'c}}}
\end{array}\right]
f(C)
\left[
\begin{array}{cc}
(E^{^{{bc}}})^*& 
(E^{^{{b'c}}})^*
\end{array}\right]
 \in {\cal BS}$$
for any holomorphic function $f$.

Recall that given $t \geq 1$, a {\it coherent} list of cylinders 
${\cal H}=(\mathrm{H}_{_{0}},\cdots,\mathrm{H}_{_{t-1}})$ 
is an ordered list of $t$ (not necessarily disjoint) cylinders $\mathrm{H}_{_{i}}=(\mathrm{B},\mathrm{C}_{_{i}})$
for $0 \leq i < t$, meaning that all cylinders have isomorphic bases (see \cite{dama,madani} for more details).

Now, given a graph $\mathrm{G}$ which is  
decomposed into $k$ edge-disjoint spanning subgraphs ${\cal G}=(\mathrm{G}_{_{0}},\cdots,\mathrm{G}_{_{t-1}})$ each of which has no isolated vertex,
along with a coherent set of cylinders 
${\cal H}=(\mathrm{H}_{_{0}},\cdots,\mathrm{H}_{_{t-1}})$, one may talk about the cylindrical construction  ${\cal G}\boxtimes{\cal H}$, which is the graph constructed 
by replacing each edge of $\mathrm{G}_{_{i}}$ by the corresponding cylinder $\mathrm{H}_{_{i}}$
for all $0 \leq i < t$ and then identifying the bases of cylinders corresponding to  each vertex of $\mathrm{G}$. Then it is not difficult to verify that the adjacency 
matrix of ${\cal G}\boxtimes{\cal H}$, denoted by $A({{\cal G}\boxtimes{\cal H}})$, can be organized as the following block matrix,
$$A({{\cal G}\boxtimes{\cal H}})= \left[
		\begin{array}{c|ccc}
		I\otimes B+\displaystyle{\sum_{i=0}^{t-1}} (G_{_i}\otimes E^{^{bb'}}_{_i})&M_{_0}&\cdots&M_{_{t-1}}\\ \hline
		{(M_{_0})}^*&I\otimes C_{_0}&&0\\ 
		\vdots &&\ddots&\\ 
		{(M_{_{t-1}})}^* &0&&I\otimes C_{_{t-1}}
		\end{array}\right].$$
\subsection{The main spectral theorem}

\begin{defin}{
For $t \geq 1$, a {\it commutative $t$-decomposition} of a graph $\mathrm{G}$ is a family of $t$ edge-disjoint spanning subgraphs of $\mathrm{G}$,
as $(\mathrm{G}_{_{0}},\cdots,\mathrm{G}_{_{t-1}})$ each of which has no isolated vertex such that 
for every pair  of disjoint indices  $i$ and $j$, the matrices  $G_{_{i}}$ and $G_{_{j}}$ commute, and moreover,
$G=\displaystyle{\sum_{i=0}^{t-1}} G_{_{i}}$.

For any such commutative $t$-decomposition, a numbering of the spectrum of each component 
$\mathrm{G}_{_{i}}$ as $\theta_i^{^j}$ for $0 \leq i < t$ and $1 \leq j \leq n$, is said to be a {\it compatible} numbering of eigenvalues if there exists a unitary matrix $U$ such that the $(j,j)$ entry of the diagonal matrix $UG_{_{i}}U^{-1}$
is $\theta_i^{^j}$.
}\end{defin}

Given an integer $n \geq 3$ and $t$ mutually distinct positive integers 
$k_{_0},\cdots,k_{_{t-1}}$ with $1 \leq k_{_{i}} < n/2$ for $0 \leq i < t$, let $\mathrm{C}_{_{n}}(k_{_0},\cdots,k_{_{t-1}})$ be the Cayley graph defined on the vertex set $\{0,1,\cdots,n-1\}$ and with the edge set 
$\cup_{_{i=1}}^{^{t}}E_{_{i}}$ where 
$$E_{_{i}} \isdef  \{uv \ | \ v=u+k_{_i}\ ({\rm mod}\ n)\}.$$
 In this setting, $\mathrm{C}_{_{n}}(k_{_{i}})$ stands for the subgraph on the edge set $E_{_{i}}$.

Now, as a very basic and important example of a compatible numbering of eigenvalues, let $U_{_F}$ be the Fourier $n\times n$ matrix whose $(r,s)$ entry is equal to ${1\over \sqrt{n}}\zeta_{_n}^{^{r\times s}}$ where $\zeta_{_n}$ is the $n$th root of unity, $\exp(\frac{2\pi {\sf i}}{n})$. Then note that, 
$U_{_F}\mathrm{C}_{_{n}}(k)U_{_F}^{^{-1}}$ is a diagonal matrix with $(r,r)$ entry equal to $2\cos(2\pi rk/n)$. Hence, $\theta^{^{j}}_{_{i}} \isdef 2\cos(2\pi jk_{_{i}}/n)$ is a compatible numbering for the eigenvalues of the decomposition $$(\mathrm{C}_{_{n}}(k_{_{0}}),\mathrm{C}_{_{n}}(k_{_{1}}),\cdots,\mathrm{C}_{_{n}}
(k_{_{t-1}})).$$ 

The following theorem links the fundamental subjects, {\it graph decompositions} and 
{\it graph spectra} in a general setup, making it possible to use results on graph decompositions to construct graphs with prescribed properties for their spectra on the one hand, and motivate interesting decomposition problems to be investigated in forthcoming research, on the other (also see  Section~\ref{sec:concluding}).

\begin{thm}\label{thm:main}
	Given $t \geq 1$, let  ${\cal G}=(\mathrm{G}_{_{0}},\cdots, \mathrm{G}_{_{t-1}})$  be a commutative $t$-decomposition of  $\mathrm{G}$ on $n$ vertices  with $|E(\mathrm{G}_{_{i}})|=m_{_{i}}$ and let 
	$\theta_i^{^j}$ be a compatible numbering of its eigenvalues.
	Also, assume that ${\cal H}=(\mathrm{H}_{_{0}},\cdots,\mathrm{H}_{_{t-1}})$ 
	is an ordered list of coherent bsymmetric cylinders with base  $\mathrm{B}$. Then
	{\footnotesize $$\phi({\cal G}\boxtimes{\cal H},x)=\left(\prod_{i=0}^{t-1}  \phi({C}_{_{i}},x)^{m_i}	
		\right)\det\left(I_{_{n}}\otimes (xI-B)-\sum_{i=0}^{t-1} \left(G_{_{i}}\otimes (E^{^{bb'}}_{_i}+R_{_{a}}(C_{_{i}}))+D_i\otimes R_{_{d}}(C_{_{i}})\right)\right)$$}
	where $D_{_i}$ is the diagonal matrix of degrees of $\mathrm{G}_{_{i}}$. 	
	In particular,
	\begin{itemize}
		\item[\rm a)]{If each ${\mathrm{G}}_{_i}$ is  a $d_{_{i}}$-regular graph, then
		{\footnotesize	$$\phi({\cal G}{\boxtimes}{\cal H},x)=\left(\prod_{i=0}^{t-1}   			
			\phi({C}_{_{i}},x)^{m_i}\right)
			\prod_{j=1}^n \phi\left(B+\sum_{i=0}^{t-1} (d_{_i}R_{_{d}}(C_{_{i}})+\theta_i^{^j}(E^{^{bb'}}_{_i}+R_{_{a}}(C_{_{i}})),x \right).$$}
	         } 
		\item[\rm b)]{If each ${\mathrm{H}}_{_i}$ has no inner-vertex, then
				{\footnotesize	$$\phi({\cal G}{\boxtimes}{\cal H},x)=\prod_{j=1}^n\phi \left(B+\sum_{i=0}^{t-1}\theta_i^{^j}E^{^{bb'}}_{_i},x\right).$$}
		
		Specifically, if also the base $\mathrm{B}$ of each ${\mathrm{H}}_{_i}$ 
		is an empty  graph, then we have the following generalized tensor product,
		{\footnotesize	$$\phi({\cal G}{\boxtimes}{\cal H},x)=\prod_{j=1}^n\phi\left(\sum_{i=0}^{t-1}\theta_i^{^j}E^{^{bb'}}_{_i},x\right).$$}
		}	 	 		 
	\end{itemize}
\end{thm}
\begin{proof}{Clearly, we have 
		$${{\cal G}{\boxtimes}{\cal H}} \isdef \left[
		\begin{array}{c|ccc}
		I_{_{n}}\otimes B+\displaystyle\sum_{i=0}^{t-1} (G_{_i}\otimes E^{^{bb'}}_{_i})&M_{_0}&\cdots&M_{_{t-1}}\\ \hline
		{(M_{_0})}^*&I_{_{m_{_{1}}}}\otimes C_{_0}&&0\\ 
		\vdots &&\ddots&\\ 
		{(M_{_{t-1}})}^* &0&&I_{_{m_{_{k}}}}\otimes C_{_{t-1}}
		\end{array}\right].$$

        In what follows, we prove identities for the characteristic polynomial of 
        ${\cal G}\boxtimes{\cal H}$ when the variable $x$ is a real number outside of the spectra of $C_{_{i}}$'s. Note that this is sufficient to show that the characteristic identities actually hold for all real numbers $x \in \mathbb{R}$
        since two polynomials of variable $x$ are identical if they are equal for all but a finite number of values for $x$.
        
         Using determinant rule for block matrices, we have
		{\footnotesize	$$
		\phi({\cal G}\boxtimes{\cal H},x)=(\prod_{i=0}^{t-1}\det(xI-C_{_{i}})^{m_i})\det(I\otimes (xI-B)-\sum_{i=0}^{t-1}(G_{_{i}}\otimes E^{^{bb'}}_{_i}+M_{_{i}}(I_{m_i}\otimes(xI-C_{_{i}}))^{-1}M_{_{i}}^*)),$$}
		and consequently,
		{\footnotesize	$$\phi({\cal G}\boxtimes{\cal H},x)=(\prod_{i=0}^{t-1}\det(xI-C_{_{i}})^{m_i})\det(I\otimes (xI-B)-\sum_{i=0}^{t-1}(G_{_{i}}\otimes E^{^{bb'}}_{_i}+\sum_{j=1}^{m_{_i}}(M_{_{ij}}(xI-C_{_{i}})^{-1}M_{_{ij}}^*))),$$}
		where $M_{_{ij}}$ is the column of $M_{_{i}}$ corresponding to the $j$th edge of $\mathrm{G}_{_{i}}$,
		with only two non-zero blocks equal to $E_{_{i}}^{^{bc}}, E_{_{i}}^{^{b'c}}$ corresponding to the two ends of the $j$th edge.
		 
	Note that, $M_{_{ij}}(xI-C_{_{i}})^{-1}M_{_{ij}}^*$ is an $n\times n$ block matrix with four non-zero block, two of which are two diagonal blocks
		 corresponding to the two ends of the $j$th edge
		 equal to $R_{_{d}}(C_{_{i}}))$ and two  other
		 blocks corresponding to the two entries of the $j$th  edge  in the adjacency matrix of $G_{_{i}}$
		  equal to $R_{_{a}}(C_{_{i}}))$.
Now, using Proposition~\ref{pro:bs} we have
		{\footnotesize	$$\phi({\cal G}\boxtimes{\cal H},x)=(\prod_{i=0}^{t-1}\det(xI-C_{_{i}})^{m_i})\det(I\otimes (xI-B)-\sum_{i=0}^{t-1}(G_{_{i}}\otimes (E^{^{bb'}}_{_i}+R_{_{a}}(C_{_{i}}))+D_i\otimes R_{_{d}}(C_{_{i}}))).$$}
		 
		Since $G_{_{i}}$'s are  symmetric matrices that commute with each other, 
		they are simultaneously diagonalizable by a unitary matrix $U$. Let  $\theta_{_i}^{^1},\cdots, \theta_{_i}^{^n}$
		be the corresponding compatible numbering of eigenvalues of 	$G_{_{i}}$'s, i.e. for every  $0 \leq i < t$ we have 
		$$UG_{_{i}}U^{-1}=\left( \begin{array}{cccc}
		\theta_{_i}^{^1} & 0&\cdots&0 \\
		0& \theta_{_i}^{^2}&\cdots&0 \\
		\vdots& &\ddots&\vdots \\
		0&0 &\cdots&\theta_{_i}^{^n}
		\end{array} \right).
		$$
		Now, if each $\mathrm{G}_{_{i}}$ is a $d_i$-regular graph, then 
        {\footnotesize		$$\phi({\cal G}\boxtimes{\cal H},x)=(\prod_{i=0}^{t-1}\det(xI-C_{_{i}})^{m_i})\det(I\otimes (xI-B-\sum_{i=0}^{t-1}d_{_i}R_{_{d}}(C_{_{i}}))-\sum_{i=0}^{t-1}(UG_{_{i}}U^{-1}\otimes (E^{^{bb'}}_{_i}+R_{_{a}}(C_{_{i}})))$$
		
		$$=(\prod_{i=0}^{t-1}\det(xI-C_{_{i}})^{m_i})
		\det(I\otimes (xI-B-\sum_{i=0}^{t-1}d_{_i}R_{_{d}}(C_{_{i}}))$$
	$$-\sum_{i=0}^{t-1}\left[\begin{array}{cccc}
		\theta_i^1(E^{^{bb'}}_{_i}+R_{_{a}}(C_{_{i}})) & 0&\cdots&0 \\
		0& \theta_i^2(E^{^{bb'}}_{_i}+R_{_{a}}(C_{_{i}}))&\cdots&0 \\
		\vdots& &\ddots&\vdots \\
		0&0 &\cdots&\theta_i^n (E^{^{bb'}}_{_i}+R_{_{a}}(C_{_{i}}))
		\end{array} \right])
		$$
		$$=(\prod_{i=0}^{t-1}\det(xI-C_{_{i}})^{m_i}) \times$$}
		{\tiny $$\det(\left[\begin{array}{ccc}
		xI-B-\displaystyle\sum_{i=0}^{t-1}(d_{_i}R_{_{d}}(C_{_{i}})+\theta_i^1(E^{^{bb'}}_{_i}+R_{_{a}}(C_{_{i}}))) &\cdots&0 \\
		\vdots &\ddots&\vdots \\
		0 &\cdots&xI-B-\displaystyle\sum_{i=0}^{t-1}(d_{_i}R_{_{d}}(C_{_{i}})+\theta_i^n (E^{^{bb'}}_{_i}+R_{_{a}}(C_{_{i}})))
		\end{array} \right])
		$$}
		Hence,
		 {\footnotesize $$\phi({\cal G}\boxtimes{\cal H},x)=(\prod_{i=0}^{t-1}\det(xI-C_{_{i}})^{m_i})\prod_{j=1}^n\det(xI-B-\displaystyle\sum_{i=0}^{t-1}(d_{_i}R_{_{d}}(C_{_{i}})+\theta_i^j(E^{^{bb'}}_{_i}+R_{_{a}}(C_{_{i}})))).$$
	    }
	    
	    	    Part $(b)$ follows directly from the main equality above.
	}\end{proof}

It is instructive to go through a couple of remarks on some special cases (see \cite{dama,madani} for the necessary background).

First, note that the identity cylinder  is clearly bsymmetric and the corresponding 
cylindrical construction is identical to the tensor product of matrices. Hence, Part~$(b)$'s reference to tensor products is justified.

On the other hand, it is known that a large number of graph products are essentially 
symmetric cylindrical constructs. In particular, one may refer to the NEPS or the lexicographic products of graphs. Hence, Theorem~\ref{thm:main} rediscovers a number of 
previously known results for such constructs.

It is interesting to recall that the Petersen graph is actually a cylindrical construct 
on the complete graph on five vertices using the coherent set of cylinders $\{\sqcap,\sqcup\}$ (see \cite{dama,madani}). But noting that $\mathrm{K}_{_{5}}$ can be decomposed into two commuting five-cycles, one may apply Theorem~\ref{thm:main} to compute the spectrum. This scenario is worked out in detail for the more general case of generalized GI-graphs in Example~\ref{exm:igraph}. 

\begin{exm}{\label{exm:subdiv}{\bf Spectrum of a subdivision} \cite{mnu}\\
	In this example we rediscover a result of V.~B.~Menuhin \cite{mnu} on the spectrum of the
	$k$th
	subdivision of a graph $\mathrm{G}$ which is clearly isomorphic to
	$\mathrm{G} \boxtimes \mathrm{P}_{_{k+2}}$. Note that in this case, the
	cylinder $\mathrm{P}_{_{k+2}}$
	is clearly bsymmetric (see Figure~\ref{fig:path}) and
	we have $\mathrm{B}=\mathrm{K}_{_{1}}$ and $\mathrm{C}=\mathrm{P}_{_{k}}$.
	Hence, one may verify that
	$$R_{_{d}}(P_{_{k}})=\frac{U_{_{k-1}}(x/2)}{U_{_{k}}(x/2)}, \ \   R_{_{a}}(P_{_{k}})=\frac{1}{U_{_{k}}(x/2)}
	\ \ {\rm and} \ \  E^{^{bb'}}=0,$$
	in which $U_{_{k}}(x)$ is the Chebyshev polynomial of the second kind.
	Therefore, using Theorem~\ref{thm:main} we have,
	{ $$	\begin{array}{ccl}
\phi(\mathrm{G} \boxtimes \mathrm{P}_{_{k+2}},x)&=&
		\phi(\mathrm{P}_{_{k}},x)^{|E(\mathrm{G})|}\
		\phi\left( G\otimes R_{_{a}}(P_{_{k}})+D\otimes R_{_{d}}(P_{_{k}}),x\right)\\
		&&\\
		&=&\left(U_{_{k}}(x/2)\right)^{|E(\mathrm{G})|}
		\phi\left( R_{_{a}}(P_{_{k}})G+R_{_{d}}(P_{_{k}}) D,x\right)\\
				&&\\
             &=&\left(U_{_{k}}(x/2)\right)^{|E(\mathrm{G})|-n}
		\phi\left( G+U_{_{k-1}}(x/2)D,xU_{_{k}}(x/2)\right).  
				\end{array}$$}
	In particular, if $\mathrm{G}$ is a $d$-regular graph, then
		{ $$	\begin{array}{ccl}
\phi(\mathrm{G} \boxtimes \mathrm{P}_{_{k+2}},x)&=&
		\left( U_{_{k}}(x/2)\right)^{|E(\mathrm{G})|}
		\phi\left(G,x-d\ \frac{U_{_{k-1}}(x/2)}{U_{_{k}}(x/2)}\right)\\
					&&\\
	&=&\left( U_{_{k}}(x/2)\right)^{|E(\mathrm{G})|-n}
		\phi\left(G,x{U_{_{k}}(x/2)}-d{U_{_{k-1}}(x/2)}\right). 				\end{array}$$}
}\end{exm}

\begin{figure}[ht]

\[\begin{tikzpicture}[x=1.5cm, y=1.5cm]
\vertex (b1) at (1.5,1) [label=below:$b_2$] {};
\vertex (b2) at (1.5,2) [label=above:$b_0$] {};
\vertex (b3) at (1,1.5) [label=below:$b_1$] {};
\vertex (b4) at (2,1.5) [label=below:$b_3$] {};
\vertex (b'1) at (3,1) [label=below:$b'_2$] {};
\vertex (b'2) at (3,2) [label=above:$b'_0$] {};
\vertex (b'3) at (2.5,1.5) [label=below:$b'_1$] {};
\vertex (b'4) at (3.5,1.5) [label=below:$b'_3$] {};

\path
(b1) edge (b2)
(b1) edge (b3)
(b1) edge (b4)
(b2) edge (b3)
(b2) edge (b4)
(b4) edge (b3)
(b'1) edge (b'2)
(b'1) edge (b'3)
(b'1) edge (b'4)
(b'2) edge (b'3)
(b'2) edge (b'4)
(b'4) edge (b'3)
(b'2) edge[bend right=20] (b2);

\end{tikzpicture}\]

\caption{\label{fig:pii}The cylinder $\sqcap_{_{4,0}}$.}	
\end{figure}
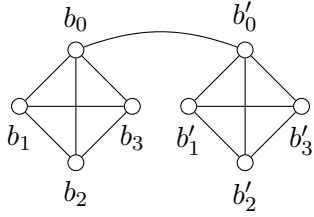

\begin{exm}{\label{exm:igraph}{\bf Spectra of GI-graphs \cite{conder}}\\
	
	Let $\sqcap_{_{t,i}}$ be the bsymmetric cylinder whose bases are isomorphic to 
	$\mathrm{K}_{_t}$ on the vertex set $\{b_{_{0}},b_{_{1}},\cdots,b_{_{t-1}}\}$ in which there is only one edge joining the $i$th vertices $b_{_{i}}$'s of the two bases (see Figure~\ref{fig:pii}). In this setting note that 
	${\cal H} \isdef \{\sqcap_{_{t,0}},\cdots,\sqcap_{_{t,t-1}}\}$ form a coherent set of bsymmetric cylinders.
	
	Now, given a positive integer $n > 0$ and $t$ positive and distinct integers $k_{_0},\cdots,k_{_{t-1}}$ each less than $n/2$, the {\it generalized I-graph}  $\mathrm{GI}(n;k_{_0},\cdots,k_{_{t-1}})$ is
	the cylindrical construct 
    $$\mathrm{GI}(n;k_{_0},\cdots,k_{_{t-1}}) \isdef 
    \mathrm{C}_{_{n}}(k_{_0},\cdots,k_{_{t-1}}) \boxtimes {\cal H}, $$
	in which the cylinder
	$\sqcap_{_{t,i}}$ is assigned to  the edge set $E_{_{i}}$. 
	
	First, note that, it is not hard to verify that this definition is equivalent to the original definition introduced in \cite{conder}. On the other hand, it must be mentioned that the special case for which $t$ is equal to $2$ gives rise to the definition of 
	I-graphs \cite{igraph} and when $t=k_{_0}=2$ we have the well-known generalized Petersen graphs \cite{Gera,gepe}.
	
	In what follows, we rediscover the characteristic polynomial and the spectra of generalized I-graphs using Theorem~\ref{thm:main}.
	
	One may verify that the subgraphs $\mathrm{C}_{_{n}}(k_{_i})$ for $0 \leq i < t$
	constitute a commutative $t$-decomposition of $\mathrm{C}_{_{n}}(k_{_0},\cdots,k_{_{t-1}})$, and consequently, Theorem~\ref{thm:main}
	can be applied. Noting that 
	$\theta_{_i}^{^j}= 2\cos(2jk_{_i}\pi/n))$, $B=J-I$, and moreover $E^{^{bb'}}_{_{i}}$ 
	is a matrix with all entries equal to zero except the $(i,i)$ entry which is equal to one, we have
		{$$\phi(\mathrm{GI}(n;k_{_0},\cdots,k_{_{t-1}}),x)=\prod_{_{j=1}}^{^n}\phi (J-I+\sum_{_{i=0}}^{^{t-1}}2\cos(2jk_{_i}\pi/n)E^{^{bb'}}_{_{i}}).$$}
	Also, focusing on the special case of I-graphs one gets,
		{\footnotesize $$\phi(\mathrm{I}(n;k,\ell),x)=
		\prod_{j=1}^n (x^2-(2\cos(2jk\pi/n)+2\cos(2j\ell\pi/n))x+(4\cos(2jk\pi/n)\cos(2j\ell\pi/n))-1).$$}
and moreover for $j=1,\ldots,n$ the eigenvalues are as follows,
$$(\cos(2jk\pi/n)+\cos(2j\ell\pi/n))\pm \sqrt{(\cos(2jk\pi/n)-\cos(2j\ell\pi/n))^2+1}.$$

}\end{exm}
	
Using notations of Example~\ref{exm:igraph},
 $\sqcap_{_{2,0}}$ and $\sqcap_{_{2,1}}$
 are referred to as $\pi$-cylinders. The most basic application of $\pi$-cylinders 
is their use to construct the celebrated Petersen graph (see Example~\ref{exm:igraph}
and \cite{dama,madani} for more on this).

\begin{exm}{\label{exm:myexample}{\bf On the role of internal vertices}\\
	
	In this example we explore the role of inner-vertices on the spectrum of the cylindrical construction. As a very simple example consider the bsymmetric cylinder 
	 $\mathrm{H}$ depicted in Figure~\ref{fig:mycylinder}, for which one may easily verify that 
     $B=\bar{I}$ and $R_{_d}(C)=R_{_a}(C)={1\over x}J$.

Now if $\mathrm{G}$ is  a $d$-regular graph on $n$ vertices, then using Theorem~\ref{thm:main} we have,
$$\phi({\cal G}\boxtimes \mathrm{H},x)=x^{^{nd/2}}
\prod_{j=1}^n\phi(\bar{I}+{d\over x}J+\theta^{^j}(I+{1\over x}J),x )$$
$$=x^{^{n(d-2)/2}}
\prod_{j=1}^n\phi(x\bar{I}+dJ+\theta^{^j}(xI+J),x^{^2} )$$
$$=x^{^{n(d-2)/2}}
\prod_{j=1}^n\det\left(\left[
\begin{array}{cc}
x^2-d-\theta^{^j}(x+1)& -x-d-\theta^{^j}\\
-x-d-\theta^{^j}&x^2-d-\theta^{^j}(x+1)
\end{array}\right] \right)$$
$$=x^{^{n(d-2)/2}} \prod_{j=1}^n\ 
x(x-(\theta^{^j}-1)
(x^{^2}-(\theta^{^j}+1)x-2(\theta^{^j}+d))).$$
Therefore, ${\cal G}\boxtimes \mathrm{H}$ has zero eigenvalues coming from the 
internal subgraph of the cylinder along with $2n$ other eigenvalues as follows,
$$(0)^{nd/2},\left\{(\theta^{^j}-1),
\left(\frac{(\theta^{^j}+1)\pm\sqrt{(\theta^{^j}+1)^2+8(\theta^{^j}+d)}}{2}\right)\right\}_{1 \leq j \leq n}.$$
This also can be true for the general case using Theorem~\ref{thm:main} when one notes that eigenvalues of the  internal subgraph of the cylinder passes to the spectrum of the 
cylindrical construct when the number of edges of the base-graph $\mathrm{G}$ is larger than the number of its vertices.
}\end{exm}

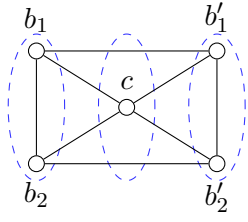
\begin{figure}[t]

\[\begin{tikzpicture}[x=1.5cm, y=1.5cm]
\vertex (b1) at (1,1) [label=below:$b_2$] {};
\vertex (b2) at (1,2) [label=above:$b_1$] {};
\vertex (c) at (1.8,1.5) [label=above:$c$] {};
\vertex (b'1) at (2.6,1) [label=below:$b'_2$] {};
\vertex (b'2) at (2.6,2) [label=above:$b'_1$] {};
\draw[color=blue!80, dashed] (1,1.5) ellipse (0.25 and 0.65);
\draw[color=blue!80, dashed] (2.6,1.5) ellipse (0.25 and 0.65);
\draw[color=blue!80, dashed] (1.8,1.5) ellipse (0.25 and 0.65);

\path
(b1) edge (b2)
(b1) edge (b'1)
(b'2) edge (b2)
(b'1) edge (b'2)
(c) edge (b'2)
(c) edge (b2)
(b1) edge (c)
(b'1) edge (c);

\end{tikzpicture}\]

\caption{\label{fig:mycylinder}See Example~\ref{exm:myexample}.}	
\end{figure}

\section{Sparsification by tree-cylinders}\label{sec:sparsifier}

Strictly speaking, for cylinders with no inner-vertices, Theorem~\ref{thm:main}
reduces the problem of computing the eigenvalues of the cylindrical construct to computing the spectrum of a perturbation of the base of the cylinders.

On the other hand, as far as finding extremely connected graphs (e.g. say Ramanujan graphs) are concerned, one should concentrate on cylindrical constructions that are capable of reducing the degree. For instance, applications such as Example~\ref{exm:igraph} 
shows such properties (also see \cite{dama,madani}). In this scenario, using such degree-reduction 
cylindrical processes, one may start from a highly connected graph (e.g. a complete graph) and apply iterative degree reductions to find highly connected graphs of smaller degree. Clearly, the main question in this setting is to estimate the rate at which the cylindrical construction expands the spectrum.

{\it Tree cylinders} are introduced in \cite{madani} as examples of cylinders that may be used 
in such degree reduction procedures. Hence, in this section we are going to delve a bit more into the spectral properties of such constructs.

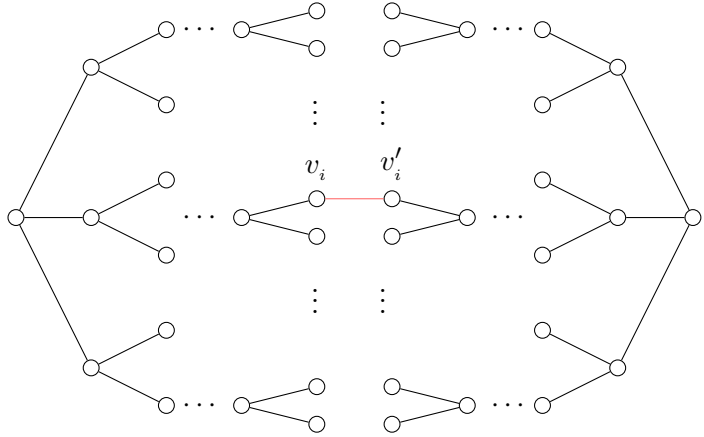
\begin{figure}[t]

 \[\begin{tikzpicture}[x=1cm, y=1cm]
 \vertex (r) at (2,3) [label=left:$$] {};
 \vertex (b1) at (3,1) [label=above:$$] {};
 \vertex (b2) at (3,3) [label=below:$$] {};
 \vertex (b3) at (3,5) [label=below:$$] {};
 \vertex (b11) at (4,0.5) [label=below:$$] {};
 \vertex (b12) at (4,1.5) [label=below:$$] {};
 \vertex (b21) at (4,2.5) [label=below:$$] {};
 \vertex (b22) at (4,3.5) [label=above:$$] {};
 \vertex (b31) at (4,4.5) [label=below:$$] {};
 \vertex (b32) at (4,5.5) [label=above:$$] {};
 \vertex (bi) at (5,0.5) [label=above:$$] {};
 \vertex (bii) at (5,3) [label=below:$$] {};
 \vertex (biii) at (5,5.5) [label=below:$$] {};
 \vertex (bi11) at (6,0.25) [label=below:$$] {};
 \vertex (bi12) at (6,0.75) [label=below:$$] {};
 \vertex (bi21) at (6,2.75) [label=below:$$] {};
 \vertex (bi22) at (6,3.25) [label=above:$v_{_i}$] {};
 \vertex (bi31) at (6,5.25) [label=below:$$] {};
 \vertex (bi32) at (6,5.75) [label=above:$$] {};

 \node (v') at (5,0.5) [label=left:$\cdots$] {};
 \node (h) at (6.3,4.5) [label=left:$\vdots$] {};
 \node (h) at (6.3,2) [label=left:$\vdots$] {};
 \node (v') at (5,3) [label=left:$\cdots$] {};
 \node (v') at (5,5.5) [label=left:$\cdots$] {};
 \vertex (r') at (11,3) [label=left:$$] {};
 \vertex (b'1) at (10,1) [label=above:$$] {};
 \vertex (b'2) at (10,3) [label=below:$$] {};
 \vertex (b'3) at (10,5) [label=below:$$] {};
 \vertex (b'11) at (9,0.5) [label=below:$$] {};
 \vertex (b'12) at (9,1.5) [label=below:$$] {};
 \vertex (b'21) at (9,2.5) [label=below:$$] {};
 \vertex (b'22) at (9,3.5) [label=above:$$] {};
 \vertex (b'31) at (9,4.5) [label=below:$$] {};
 \vertex (b'32) at (9,5.5) [label=above:$$] {};
 \vertex (b'i) at (8,0.5) [label=above:$$] {};
 \vertex (b'ii) at (8,3) [label=below:$$] {};
 \vertex (b'iii) at (8,5.5) [label=below:$$] {};
 \vertex (b'i11) at (7,0.25) [label=below:$$] {};
 \vertex (b'i12) at (7,0.75) [label=below:$$] {};
 \vertex (b'i21) at (7,2.75) [label=below:$$] {};
 \vertex (b'i22) at (7,3.25) [label=above:$v'_{_i}$] {};
 \vertex (b'i31) at (7,5.25) [label=below:$$] {};
 \vertex (b'i32) at (7,5.75) [label=above:$$] {};

 \node (v') at (9.1,0.5) [label=left:$\cdots$] {};
 \node (h) at (7.2,4.5) [label=left:$\vdots$] {};
 \node (h) at (7.2,2) [label=left:$\vdots$] {};
 \node (v') at (9.1,3) [label=left:$\cdots$] {};
 \node (v') at (9.1,5.5) [label=left:$\cdots$] {};

 \path 
 (r) edge (b1)
 (r) edge (b2)
 (r) edge (b3)
 (b11) edge (b1)
 (b12) edge (b1)
 (b21) edge (b2)
 (b22) edge (b2)
 (b31) edge (b3)
 (b32) edge (b3)
 (bi11) edge (bi)
 (bi12) edge (bi)
 (bi21) edge (bii)
 (bi22) edge (bii)
 (bi31) edge (biii)
 (bi32) edge (biii)
 (r') edge (b'1)
 (r') edge (b'2)
 (r') edge (b'3)
 (b'11) edge (b'1)
 (b'12) edge (b'1)
 (b'21) edge (b'2)
 (b'22) edge (b'2)
 (b'31) edge (b'3)
 (b'32) edge (b'3)
 (b'i11) edge (b'i)
 (b'i12) edge (b'i)
 (b'i21) edge (b'ii)
 (b'i22) edge (b'ii)
 (b'i31) edge (b'iii)
 (b'i32) edge (b'iii)
 (b'i22) edge[color=red!60] (bi22);
 \end{tikzpicture}\]
 
\caption{\label{fig:treecylinder1} The graph ${\mathrm{T}}^{^{\bullet}}_{_{3,h,i}}$ (See Definition~\ref{def:treecyl}).}	
\end{figure}

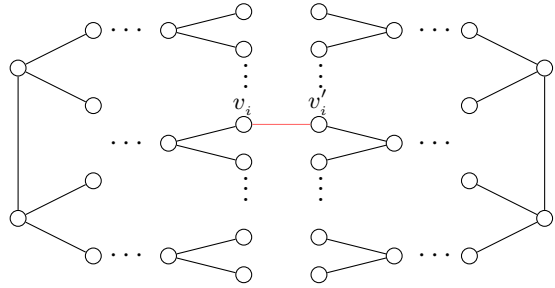
\begin{figure}[t]

 \[\begin{tikzpicture}[x=1cm, y=1cm]
	\vertex (r0) at (2,3) [label=left:$$] {};
	\vertex (r1) at (2,5) [label=left:$$] {};
	\vertex (b01) at (3,2.5) [label=above:$$] {};
	\vertex (b02) at (3,3.5) [label=below:$$] {};
	\vertex (b11) at (3,4.5) [label=below:$$] {};
	\vertex (b12) at (3,5.5) [label=below:$$] {};
	\vertex (bi) at (4,5.5) [label=below:$$] {};
	\vertex (bi1) at (5,5.75) [label=below:$$] {};
	\vertex (bi2) at (5,5.25) [label=below:$$] {};
	\vertex (bj) at (4,2.5) [label=below:$$] {};
	\vertex (bj1) at (5,2.75) [label=below:$$] {};
	\vertex (bj2) at (5,2.25) [label=below:$$] {};
	\vertex (b3) at (4,4) [label=below:$$] {};
	\vertex (b31) at (5,3.75) [label=below:$$] {};
	\vertex (b32) at (5,4.25) [label=below:$$] {};
	\node (c) at (5,4.1) [label=above:{\footnotesize${v_{_i}}$}] {};

\node (v) at (5.35,5) [label=left:$\vdots$] {};
\node (v) at (5.35,3.5) [label=left:$\vdots$] {};
\node (v) at (4,2.5) [label=left:$\cdots$] {};
\node (v) at (4,5.5) [label=left:$\cdots$] {};
\node (v) at (4,4) [label=left:$\cdots$] {};

	\vertex (r'0) at (9,3) [label=left:$$] {};
	\vertex (r'1) at (9,5) [label=left:$$] {};
	\vertex (b'01) at (8,2.5) [label=above:$$] {};
	\vertex (b'02) at (8,3.5) [label=below:$$] {};
	\vertex (b'11) at (8,4.5) [label=below:$$] {};
	\vertex (b'12) at (8,5.5) [label=below:$$] {};
	\vertex (b'i) at (7,5.5) [label=below:$$] {};
	\vertex (b'i1) at (6,5.75) [label=below:$$] {};
	\vertex (b'i2) at (6,5.25) [label=below:$$] {};
	\vertex (b'j) at (7,2.5) [label=below:$$] {};
	\vertex (b'j1) at (6,2.75) [label=below:$$] {};
	\vertex (b'j2) at (6,2.25) [label=below:$$] {};
	\vertex (b'3) at (7,4) [label=below:$$] {};
	\vertex (b'31) at (6,3.75) [label=below:$$] {};
	\vertex (b'32) at (6,4.25) [label=below:$$] {};
	\node (c') at (6,4.1) [label=above:{\footnotesize${v'_{_i}}$}] {};

\node (v) at (6.35,5) [label=left:$\vdots$] {};
\node (v) at (6.35,3.5) [label=left:$\vdots$] {};
\node (v') at (8.1,2.5) [label=left:$\cdots$] {};
\node (v') at (8.1,5.5) [label=left:$\cdots$] {};
\node (v') at (8.1,4) [label=left:$\cdots$] {};

	\path 
		(r0) edge (r1)
		(b01) edge (r0)
		(b02) edge (r0)
		(r1) edge (b11)
		(r1) edge (b12)
		(bi1) edge (bi)
		(bi2) edge (bi)
		(bj1) edge (bj)
		(bj2) edge (bj)
		(b31) edge (b3)
		(b32) edge (b3)
		(r'0) edge (r'1)
		(b'01) edge (r'0)
		(b'02) edge (r'0)
		(r'1) edge (b'11)
		(r'1) edge (b'12)
		(b'i1) edge (b'i)
		(b'i2) edge (b'i)
		(b'j1) edge (b'j)
		(b'j2) edge (b'j)
		(b'31) edge (b'3)
		(b'32) edge (b'3)
		(b'32) edge[color=red!60] (b32);
\end{tikzpicture}\]
\caption{\label{fig:treecylinder2} The graph $\overline{\mathrm{T}}_{_{3,h,i}}$ (See Definition~\ref{def:treecyl}).}	
\end{figure}

\begin{defin}{\label{def:treecyl}{\bf Tree cylinders} \cite{madani}
	
	A {\it tree-cylinder} in general is a cylinder whose bases are identical vertex-labeled 
	copies of a tree for which there exists only a unique edge joining one leaf of the right base to another leaf of the left one. In what follows we will only be dealing 
	with bsymmetric regular tree-cylinders, 
	${\mathrm{T}}^{^{\bullet}}_{_{k,h,i}}$ and $\overline{\mathrm{T}}_{_{k,h,i}}$ defined as,
	\begin{itemize}
	\item ${\mathrm{T}}^{^{\bullet}}_{_{k,h,i}}$: For $0 \leq i < k\times (k-1)^{^{h-1}}$,  ${\mathrm{T}}^{^{\bullet}}_{_{k,h,i}}$, is a cylinder whose bases are identical copies of the {\it complete rooted} $k$-regular tree 
	${\mathrm{T}}^{^{\bullet}}_{_{k,h}}$ of height $h$ in which the leaves are numbered consecutively and there is only one edge joining the $i$th leaves of both bases (see Figure~\ref{fig:treecylinder1}). 
	
	\item $\overline{\mathrm{T}}_{_{k,h,i}}$: For $0 \leq i < 2(k-1)^{^{h}}$, $\overline{\mathrm{T}}_{_{k,h,i}}$, is a cylinder whose bases are identical copies of the {\it complete unrooted} $k$-regular tree 
	$\overline{\mathrm{T}}_{_{k,h}}$ of height $h$ in which the leaves are numbered consecutively and there is only one edge joining the $i$th leaves of both bases (see Figure~\ref{fig:treecylinder2}). 
	\end{itemize}

    Note that, by definition, both ${\mathrm{T}}^{^{\bullet}}_{_{k,h,i}}$ and $\overline{\mathrm{T}}_{_{k,h,i}}$ are 
	 bsymmetric cylinders.
    As shown in \cite{madani}, tree-cylinders can be used as sparsifiers, in the sense that they can be used in a cylindrical construction that reduces the degree and increases 
    the girth of the base graph. 
}\end{defin}

	Hereafter, we always assume that the corresponding adjacency matrices
	${T}^{^{\bullet}}_{_{k,h}}$ and $\overline{T}_{_{k,h}}$ are formed
	based  on the numbering for which there is a priority for lower levels and at each level the vertices appear consecutively from left to the right. Note that in this setting the adjacency matrix ${T}^{^{\bullet}}_{_{k,h}}$ can also be defined recursively as 
$${T}^{^{\bullet}}_{_{k,h}}=\left[
\begin{array}{cc}
{T}^{^{\bullet}}_{_{k,h-1}}&C_{_h}\\ 
C_{_h}^*&0
\end{array}\right],$$
in which $C_{_h}$ is the adjacency submatrix corresponding to the leaves of 
${\mathrm{T}}^{^{\bullet}}_{_{k,h}}$. Clearly, there is a similar recursive formula for $\overline{T}_{_{k,h}}$ too.

Since in what follows we will be needing to refer to some special cyclic decompositions of complete graphs, let us consider a complete graph $\mathrm{K}_{_{n}}$ on the vertex set $V(\mathrm{K}_{_{n}})\isdef \{0,1,\cdots,n-1\}$ along with its spanning subgraphs 
$\mathrm{C}_{_{n}}(k)$ already defined in Example~\ref{exm:igraph}.
Note that if $n$ and $k$ 
are relatively prime integers, then 
 $\mathrm{C}_{_{n}}(k)$ is a cycle. Also, if $n > 2$ is a prime integer, then the set containing $t=\frac{n-1}{2}$ of such subgraphs for different 
values of $1 \leq k \leq t$, constitute a commutative $t$-decomposition of $\mathrm{K}_{_{n}}$.

\begin{defin}{{\bf tree-mixing} \label{def:btmixing}
	
	Consider $\mathrm{T}$, a $3$-regular tree of height $h$, and assume that we have assigned extra labels $a_{_{i}} \in \mathbb{F}[x]$ for 
	the $i$th leaf at the $h$th level for $0 \leq i \leq \Lambda_{_{h}}-1$,
	where $\Lambda_{_{\ell}}$ is the number of vertices at the $\ell$th level. A 
	{\it tree mixing} of the ordered list $(a_{_{0}},a_{_{1}},\cdots,a_{_{\Lambda_{_{h}}-1}})$ is a recursive process of assigning labels
	in $\mathbb{F}(x)$ (the field of rational fractions of $\mathbb{F}[x]$) to the vertices of $\mathrm{T}$ starting from the leaves towards the root, in such a way that the label assigned to the vertex $v$ whose siblings have labels $b$
	and $c$ is equal to $x-(b^{-1}+c^{-1})$(note that for a rooted tree $T$,the label assigned to the root $v$ is equal to $x-(b^{-1}+c^{-1}+d^{-1})$ where the siblings of $v$ have labels $b$, $c$ and $d$). 
	In this setting, the product of the labels of vertices at each level $0 \leq \ell \leq h$ is called the {\it resultant} of that level denoted by $R_{_{\mathrm{T}}}(\ell)$,
	and the product of resultants of all levels, 
	$$R(\mathrm{T})\isdef\displaystyle{\prod_{\ell=0}^{h}}\ R_{_{\mathrm{T}}}(\ell),$$ is  
	defined to be the {\it resultant} of the labeled tree 
	$\mathrm{T}(a_{_{0}},a_{_{1}},\cdots,a_{_{\Lambda_{_{h}}-1}})$ itself as the final result of the tree-mixing.
}\end{defin}

Now, let us concentrate on an application of Theorem~\ref{thm:main} for a very special case using tree-cylinders to generate sparse $3$-regular graphs, which also can be considered as a variant of Example~\ref{exm:igraph}. A byproduct of the following proof is the fact that in some special cases the resultant of a tree-mixing is actually a polynomial in $\mathbb{F}[x]$. 

\begin{pro}\label{pro:treecyl}

Given integers $n,t$ and $1 \leq k_{_{i}} \leq \frac{n-1}{2}$ for $0 \leq i < t$,
fix ${\cal C} \isdef (\mathrm{C}_{_{n}}(k_{_0}),\cdots,\mathrm{C}_{_{n}}(k_{_{t-1}}))$
as a $t$-decomposition of $\mathrm{C}_{_{n}}(k_{_0},\cdots,k_{_{t-1}})$.

\begin{itemize}
\item[{\rm a)}] Let $h \geq 1$ be an integer, $t= 2^{h+1}$, and
$\overline{{\cal H}}\isdef(\overline{\mathrm{T}}_{_{3,h,0}},\overline{\mathrm{T}}_{_{3,h,1}},\cdots,\overline{\mathrm{T}}_{_{3,h,t-1}}).$
 If $\overline{\Theta}_{_j}$ is a diagonal matrix whose diagonal entries are zero for any vertex which is not a leaf and is equal to $2\cos \frac{2\pi j k_{_{i}}}{n}$ for the $i$th leaf, then
$$\phi({\cal C}{\boxtimes}\overline{{\cal H}},x)=\prod_{j=1}^n \phi \left(\overline{T}_{_{3,h}}+ \overline{\Theta}_{_j},x\right)=\prod_{j=1}^n R(\overline{T}^{^{j}}_{_{3,h}}) \frac{R_{_{\overline{T}^{^{j}}_{_{3,h}}}}(0)-1}
{R_{_{\overline{T}^{^{j}}_{_{3,h}}}}(0)}$$
 in which ${\overline{T}^{^{j}}_{_{3,h}}}$ is an unrooted binary labeled tree (see Definition~{\rm \ref{def:treecyl}}) whose $i$th leaf is labeled by $\left(x-2\cos \frac{2\pi j k_{_{i}}}{n}\right).$

\item[{\rm b)}] Let $h \geq 1$ be an integer, $t=3\times 2^{h-1}$, and
${\cal H}^{^{\bullet}}\isdef(\mathrm{T}^{^{\bullet}}_{_{3,h,0}},\mathrm{T}^{^{\bullet}}_{_{3,h,1}},\cdots,\mathrm{T}^{^{\bullet}}_{_{3,h,t-1}}).$
 If $\Theta^{^{\bullet}}_{_j}$ is a diagonal matrix whose diagonal entries are zero for any vertex which is not a leaf and is equal to $2\cos \frac{2\pi j k_{_{i}}}{n}$ for the $i$th leaf, then
$$\phi({\cal C}{\boxtimes}{\cal H}^{^{\bullet}},x)=\prod_{j=1}^n \phi \left({T^{^{\bullet}}}_{_{3,h}}+ \Theta^{^{\bullet}}_{_j},x\right)=\prod_{j=1}^nR(T^{^{\bullet j}}_{_{3,h}})$$
 in which ${\mathrm{T}^{^{\bullet j}}_{_{3,h}}}$ is a rooted binary labeled tree (see Definition~{\rm \ref{def:treecyl}}) whose $i$th leaf is labeled by $\left(x-2\cos \frac{2\pi j k_{_{i}}}{n}\right).$

\end{itemize}
\end{pro}
\begin{proof}{The proof is essentially a consequence of Theorem\ref{thm:main}(b)
	along with the following fact on the existence of a recursive procedure 
	to compute the target determinant which is equal to tree-mixing. 
	To see this, let $\mathrm{T}_{_{h}}$ be a binary tree of height $h$ and consider the following target determinant for a diagonal matrix $\Theta_{_j}$,
$$\det(xI-(T_{_{h}}+ \Theta_{_j})=\det\left[
\begin{array}{cc}
xI-T_{_{h-1}}&C_{_h}\\ 
C_{_h}^*&xI-\Theta_{_j}
\end{array}\right],$$
which is  equivalent to
 $$\det(xI-\Theta_{_j})\det(xI-T_{_{h-1}}-C_{_h}(xI-\Theta_{_j})^{^{-1}}C_{_h}{^*}),$$
 assuming that $x$ is not in the spectrum of $\Theta_{_j}$.
Then, it is not hard to verify that $C_{_h}(xI-\Theta_{_j})^{^{-1}}C_{_h}^*$ is also a 
diagonal matrix whose diagonal entries are the sum of two consecutive entries of $(xI-\Theta_{_j})^{^{-1}}$. This, clearly, shows that the recursive procedure matches the tree-mixing process for the initial state labeling by the diagonal of $xI-\Theta_{_j}$ where the only difference for cases $(a)$ and $(b)$ lies in the last stage.

Note that using this recursion one ends up with a polynomial which is equivalent to the characteristic polynomial of the cylindrical construct outside of a set of finite size.
This, implies that the polynomial is, actually, the characteristic polynomial itself. 
}\end{proof}

In the next proposition the $n$th term is explained in more detail,
where the parts of the methods are
 borrowed from a similar study in \cite{strang}.

\begin{pro}\label{pro:uniformterm}

Let $p_{_{n}}(x)$ be a sequence of polynomials which are defined recursively as,
$$p_{_{0}}(x)=1, p_{_{1}}(x)=x-2, \ {\rm and} \  p_{_{n+2}}(x)=xp_{_{n+1}}(x)-2p_{_{n}}(x).$$
Then,
\begin{itemize}
\item [{\rm a)}] We have 
$$p_{_{n}}(x)=\frac{1}{2(\alpha-\beta)}(\frac{1}{\beta^{^{n+1}}}-\frac{1}{\alpha^{^{n+1}}}-2(\frac{1}{\beta^{^{n}}}-\frac{1}{\alpha^{^{n}}}))$$
in which  $\alpha$ and $\beta$ are the roots of $1-xt+2t^{^{2}}=0$ defined as 
$$\alpha=\frac{x+\sqrt{x^2-8}}{4}, \quad \beta=\frac{x-\sqrt{x^2-8}}{4}.$$

\item [{\rm b)}] The set $\{p_{_{n}}(x)\}_{_{n>1}}$ constitute a family of orthogonal
polynomials. Roots of $p_{_{n-1}}(x)$ interlaces those of $p_{_{n}}(x)$
for $n > 1$, and moreover, if $$\mu^{^{n}}_{_{1}}\geq\mu^{^{n}}_{_{2}}\geq\cdots\geq\mu^{^{n}}_{_{n}}$$
 are roots of $p_{_{n}}(x)$ and we define $\eta^{^{n-1}}_{_{j}}=2\cos(\frac{j\pi}{n})$, for $j=1,2,\ldots,n-1$, then
$$\mu^{^{n}}_{_{1}}\geq\eta^{^{n-1}}_{_{1}}\geq\mu^{^{n}}_{_{2}}\geq\eta^{^{n-1}}_{_{2}}\geq\cdots\geq\eta^{^{n-1}}_{_{n-1}}\geq\mu^{^{n}}_{_{n}}.$$

\item [{\rm c)}] Within the setup of Proposition~{\rm \ref{pro:treecyl}} we have 
$$\phi \left(\overline{T}_{_{3,h}}+ \overline{\Theta}_{_n},x\right)= (\displaystyle\prod_{i=1}^{h}p_{_{i}}(x)^{^{\frac{\overline\Lambda_{_{h-i+1}}}{2}}})(p_{_{h+1}}(x)^{^{2}}-p_{_{h}}(x)^{^{2}}),$$ 
and 
$$\phi \left(T^{^{\bullet}}_{_{3,h}}+ \Theta^{^{\bullet}}_{_n},x\right)=(\displaystyle\prod_{i=1}^{h-1}p_{_{i}}(x)^{^{\frac{\Lambda_{_{h-i+1}}^{^{\bullet}}}{2}}})p_{_{h}}(x)^{^{2}}(xp_{_{h}}(x)-3p_{_{h-1}}(x)),$$ 
where 
${\overline\Lambda}_{_{\ell}}$ and $\Lambda_{_{\ell}}^{^{\bullet}}$ are the number of vertices in $\ell$th levels of   $\overline{T}_{_{3,h}}$ and ${T}^{^{\bullet}}_{_{3,h}}$, respectively.

\end{itemize}

\end{pro}
\begin{proof}{To prove part (a) let $G(x,t)\isdef\displaystyle\sum_{n=0}^\infty p_{_{n}}(x)t^{^{n}}$. Then,
	$$p_{_{n+2}}(x)t^{^{n}}=xp_{_{n+1}}(x)t^{^{n}}-2p_{_{n}}(x)t^{^{n}}$$
	implies that
	$$\frac{1}{t^{^2}}\displaystyle\sum_{n=0}^\infty p_{_{n+2}}(x)t^{^{n+2}}=\frac{x}{t}
	\displaystyle\sum_{n=0}^\infty
	p_{_{n+1}}(x)t^{^{n+1}}-2\displaystyle\sum_{n=0}^\infty
	p_{_{n}}(x)t^{^{n}}$$
	and consequently,
	$$\frac{1}{t^{^{2}}}(G(x,t)-1-(x-2)t)=\frac{x}{t}(G(x,t)-1)-2G(x,t)$$
	showing that
	$$(\frac{1}{t^{^{2}}}-\frac{x}{t}+2)G(x,t)=\frac{1}{t^{^{2}}}-\frac{2}{t}$$
	and
	$$G(x,t)=\frac{1-2t}{(1-xt+2t^{^{2}})}.$$
	Now if $1-xt+2t^{^{2}}=0$, then
	$$\alpha=\frac{x+\sqrt{x^2-8}}{4}, \beta=\frac{x-\sqrt{x^2-8}}{4}$$
	and we have,
	$$G(x,t)=\frac{1-2t}{2(\alpha-\beta})(\frac{1}{t-\alpha}-\frac{1}{t-\beta}).$$
	Therefore,
	$$G(x,t)=\frac{1-2t}{2(\alpha-\beta)}(\frac{1}{\beta}\displaystyle\sum_{n=0}^\infty\frac{t^{^{n}}}{\beta^{^{n}}}-\frac{1}{\alpha}\displaystyle\sum_{n=0}^\infty\frac{t^{^{n}}}{\alpha^{^{n}}})$$
	where the coefficient of $t^{^n}$, $p_{_{n}}(x)$, is
	$$p_{_{n}}(x)=\frac{1}{2(\alpha-\beta)}(\frac{1}{\beta^{^{n+1}}}-\frac{1}{\alpha^{^{n+1}}}-2(\frac{1}{\beta^{^{n}}}-\frac{1}{\alpha^{^{n}}})).$$
	
For part (b), first note that 
the sequence  of polynomials $\{p_{_{n}}(x)\}_{_{n > 1}}$ satisfy a three-term recurrence relation, and consequently, by Favard's theorem they constitute an  orthogonal sequence of polynomials for which
the roots of $p_{_{n-1}}(x)$ interlace those of $p_{_{n}}(x)$ \cite{orth}.

On the other hand, note that
 $p_{_{n}}(x)$ is equal to the characteristic polynomial of the following $n$ by $n$ tridiagonal  matrix 
$$
\tau(n)\isdef\left[		\begin{array}{cccc}
2&1&&\\ 
2&0&\ddots&\\ 
&\ddots&\ddots&1\\ 
&&2&0
\end{array}\right]$$
which is similar to 
$$
{\hat\tau}(n)\isdef\left[\begin{array}{cccc}
2&\sqrt{2}&&\\ 
\sqrt{2}&0&\ddots&\\ 
&\ddots&\ddots&\sqrt{2}\\ 
&&\sqrt{2}&0
\end{array}\right].$$

Hence, if we consider the principal submatrix, $\hat\tau_{_{1,1}}(n)$, obtained by removing the first row and the first column, then its characteristic polynomial is equal to $(\sqrt{2})^{^{n-1}}U_{_{n-1}}(\frac{x}{2\sqrt{2}})$
and  this implies that the eigenvalues of $\hat\tau_{_{1,1}}$ are equal to $\eta^{^{n-1}}_{_{j}}=2\cos(\frac{j\pi}{n})$, for $j=1,2,\ldots,n-1$.
Hence, by the interlacing theorem we have 
$$\mu^{^{n}}_{_{1}}\geq\eta^{^{n-1}}_{_{1}}\geq\mu^{^{n}}_{_{2}}\geq\eta^{^{n-1}}_{_{2}}\geq\cdots\geq\eta^{^{n-1}}_{_{n-1}}\geq\mu^{^{n}}_{_{n}}.$$

It is also instructive to note that, as a well-known fact, the eigenvalues of a tridiagonal matrix with positive off-diagonal entries are  real and simple (e.g. see \cite{tridia}).
Therefore, by considering the principal submatrix obtained by removing the last row and the last column, again we rediscover the fact that the roots of $p_{_{n-1}}(x)$  interlace those of $p_{_{n}}(x)$.

To prove part (c), by using part (a) of Propsition~\ref{pro:treecyl} we have

$$\phi \left(\overline{T}_{_{3,h}}+ \overline{\Theta}_{_n},x\right)=R(\overline{T}^{^{n}}_{_{3,h}}) \frac{R_{_{\overline{T}^{^{n}}_{_{3,h}}}}(0)-1}
{R_{_{\overline{T}^{^{n}}_{_{3,h}}}}(0)}.$$
By the definition of the resultant of labeled tree, $R(\overline{T}^{^{n}}_{_{3,h}})=\displaystyle{\prod_{\ell=0}^{h}}\ R_{_{\overline{T}^{^{n}}_{_{3,h}}}}(\ell).$ 
Note that in $h$th level the labels of leaves of  $\overline{T}_{_{3,h}}$ are equal to $p_{_{1}}(x)=x-2$, and hence  $R_{_{\overline{T}^{^{n}}_{_{3,h}}}}(h)=p_{_{1}}(x)^{^{{\overline\Lambda}_{_{h}}}}$. 
The labels of vertices of  $\overline{T}^{^{n}}_{_{3,h}}$ in $(h-1)$th levels are equal to $$x-(\frac{1}{p_{_{1}}(x)}+\frac{1}{p_{_{1}}(x)})=\frac{p_{_{2}}(x)}{p_{_{1}}(x)},$$
and with the similar computation we obtain that the labels of vertices of  $\overline{T}^{^{n}}_{_{3,h}}$ in $(h-i)$th levels are equal
$\frac{p_{_{i+1}}(x)}{p_{_{i}}(x)}.$
 Consequently, $R_{_{\overline{T}^{^{n}}_{_{3,h}}}}(h-i)=(\frac{p_{_{i+1}}(x)}{p_{_{i}}(x)})^{^{{\overline\Lambda}_{_{h-i}}}}.$
Now, since ${\overline\Lambda}_{_{\ell}}=2{\overline\Lambda}_{_{\ell-1}}$, we can conclude that
 $$R(\overline{T}^{^{n}}_{_{3,h}}) \frac{R_{_{\overline{T}^{^{n}}_{_{3,h}}}}(0)-1}
 {R_{_{\overline{T}^{^{n}}_{_{3,h}}}}(0)}=(\displaystyle\prod_{i=1}^{h}p_{_{i}}(x)^{^{\frac{\overline\Lambda_{_{h-i+1}}}{2}}})(p_{_{h+1}}(x)^{^{2}}-p_{_{h}}(x)^{^{2}}).$$
The similar fact for the rooted tree $T^{^{\bullet}}_{_{3,h}}$ can be obtained analogously.
}\end{proof}

\begin{exm}{\label{exm:coxeter}{\bf The Coxeter cylinder}
	
	The cylinder $\mathrm{T}^{^\bullet}_{_{3,1,1}}$ is defined as the Coxeter cylinder in \cite{dama,madani},
	the main reason being the fact that the well-known Coxeter graph can be represented as the cylindrical construction ${\cal K}_{_{7}}\boxtimes{\cal H}$ in which 
	$${\cal H}\isdef \{\mathrm{T^{^\bullet}}_{_{3,1,i}} \ | \ 0 \leq i \leq 2\} \ \ {\rm and } \ \ 
	{\cal K}_{_{7}} \isdef (\mathrm{C}_{_{7}}(1),\mathrm{C}_{_{7}}(2),\mathrm{C}_{_{7}}(3)).$$

	For this case, note that by Proposition~\ref{pro:treecyl}(b) we have 
	$h=1, t=3$ and $n=7$ where  the characteristic polynomial is as follows,
	$$	\prod_{j=1}^7\det\left[
		\begin{array}{cccc}
		x&-1&-1&-1\\ 
		-1&x-2\cos(\frac{2\pi j}{7})&0&0\\ 
		-1&0&x-2\cos(\frac{4\pi j}{7})&0\\ 
		-1&0&0&x-2\cos(\frac{6\pi j}{7})
		\end{array}\right]$$
		$$=(x-3)(x -2 )^2 (x +1) ((x-2)(x+1)(x^2+2x-1))^6.$$

(For the computation, note that $2\cos(\frac{2\pi j}{7})=\zeta_{_7}^{^j}+\zeta_{_7}^{^{-j}}$ and that the sum of $n$th roots of unity is $0$.)

}\end{exm}

\subsection{A class of highly symmetric graphs}\label{sec:highsymmetry}

The contents of this section is essentially a generalization of Example~\ref{exm:coxeter}
where we introduce a class of highly symmetric graphs of type already introduced in Proposition~\ref{pro:treecyl} for which we have a relatively large multiplicity for the eigenvalues.

\begin{defin}{\label{def:grouplabeling}{\bf Cyclic group labeling of binary trees}\\
	
	Consider $\mathrm{T}_{_{h}}$, a complete binary tree of height $h$, of the binary monoid using concatenation, which is a labeled rooted tree with the root $r$ (representing the word of length zero at level zero)
	and $h$ levels of vertices where there are $2^\ell$ vertices at the $\ell$th level, and 
	moreover, the vertices of this complete binary tree are labeled using concatenation with 
	$0$ and $1$ according to a left-justified rule, where
	for each vertex with the label $s$, its left sibling has the label $s0$ and its right sibling has the label $s1$ (see Figure~\ref{fig:treegroup}). Note that, in this way, all binary numbers of $\ell$ bits appear at the $\ell$th level in their natural order. Hereafter, we always assume that the vertex set of a complete binary tree  of height $h$ is the set of all binary strings of length less than or equal to $h$,
	where the vertex $s$ is a sibling of a vertex $\tilde{s}$ if 
	and only if $|\tilde{s}|=|s|+1$ and $s$ is a prefix of $\tilde{s}$.
	
	Moreover, if $s=s_{_{m}} \cdots s_{_{0}}$ is a binary string of length $m+1$, then its {\it reverse } is defined as 
	$\hat{s} \isdef s_{_{0}} \cdots s_{_{m}}$ and its {\it corresponding integer} is defined as 
	$$(s)_{_{2}} \isdef s_{_{m}} 2^m + \cdots +s_{_{1}} 2 + s_{_{0}}.$$ 
	
	Let $\Gamma = \langle a \ | \ a^N=1\rangle $ be a multiplicative cyclic group of order $N$ generated by $a$, and let ${\cal C}(\Gamma)$ be the whole class 
	of cosets of subgroups of $\Gamma$. Then, the mapping 
	$\gamma:V(\mathrm{T}_{_{h}}) \to {\cal C}(\Gamma)$ defined 
	as
	$$\gamma(s) \isdef a^{(\hat{s})_{_{2}}}\langle a^{2^{^{|s|}}} \rangle$$
	is called a cyclic group labeling of $\mathrm{T}_{_{h}}$ by 
	$\Gamma$ (See Figure~\ref{fig:treegroup}) for an example of such labeling).
	
}\end{defin}

\begin{figure}[t]

 \[\begin{tikzpicture}
 \vertex (r) at (3.5,3) [label=left:$$] {};
 \vertex (r1) at (1.5,2) [label=above:\scriptsize$0$] {};
 \vertex (r2) at (5.5,2) [label=above:\scriptsize$1$] {};
 \vertex (b11) at (0.5,1) [label=above:\scriptsize$00$] {};
 \vertex (b12) at (2.5,1) [label=above:\scriptsize$01$] {};
 \vertex (b21) at (4.5,1) [label=above:\scriptsize$10$] {};
 \vertex (b22) at (6.5,1) [label=above:\scriptsize$11$] {};
 \vertex (b111) at (0,0) [label=below:\scriptsize$000$] {};
 \vertex (b112) at (1,0) [label=below:\scriptsize$001$] {};
 \vertex (b121) at (2,0) [label=below:\scriptsize$010$] {};
 \vertex (b122) at (3,0) [label=below:\scriptsize$011$] {};
 \vertex (b211) at (4,0) [label=below:\scriptsize$100$] {};
 \vertex (b212) at (5,0) [label=below:\scriptsize$101$] {};
 \vertex (b221) at (6,0) [label=below:\scriptsize$110$] {};
 \vertex (b222) at (7,0) [label=below:\scriptsize$111$] {};


 \path 
 (r) edge (r1)
 (r) edge (r2)
 (b21) edge (r2)
 (b22) edge (r2)
 (b12) edge (r1)
 (b11) edge (r1)
 (b111) edge (b11)
 (b112) edge (b11)
 (b121) edge (b12)
 (b122) edge (b12)
 (b211) edge (b21)
 (b212) edge (b21)
 (b221) edge (b22)
 (b222) edge (b22);
 \vertex (r) at (11.5,3) [label=left:\tiny$\langle a\rangle$] {};
 \vertex (r1) at (9.5,2) [label=above:\tiny$\langle a^{^2}\rangle$] {};
 \vertex (r2) at (13.5,2) [label=above:\tiny$a\langle a^{^2}\rangle$] {};
 \vertex (b11) at (8.5,1) [label=above:\tiny$\langle a^{^4}\rangle$] {};
 \vertex (b12) at (10.5,1) [label=above:\tiny$\,\,\,\,\,\,a^{^2}\langle a^{^4}\rangle$] {};
 \vertex (b21) at (12.5,1) [label=above:\tiny$a\langle a^{^4}\rangle\,\,\,\,\,\,$] {};
 \vertex (b22) at (14.5,1) [label=above:\tiny$\,\,\,\,\,\,a^{^3}\langle a^{^4}\rangle$] {};
 \vertex (b111) at (8,0) [label=below:\tiny$\langle a^{^8}\rangle$] {};
 \vertex (b112) at (9,0) [label=below:\tiny$a^{^4}\langle a^{^8}\rangle$] {};
 \vertex (b121) at (10,0) [label=below:\tiny$a^{^2}\langle a^{^8}\rangle$] {};
 \vertex (b122) at (11,0) [label=below:\tiny$a^{^6}\langle a^{^8}\rangle$] {};
 \vertex (b211) at (12,0) [label=below:\tiny$a\langle a^{^8}\rangle$] {};
 \vertex (b212) at (13,0) [label=below:\tiny$a^{^5}\langle a^{^8}\rangle$] {};
 \vertex (b221) at (14,0) [label=below:\tiny$a^{^3}\langle a^{^8}\rangle$] {};
 \vertex (b222) at (15,0) [label=below:\tiny$a^{^7} \langle a^{^8}\rangle$] {};


 \path 
 (r) edge (r1)
 (r) edge (r2)
 (b21) edge (r2)
 (b22) edge (r2)
 (b12) edge (r1)
 (b11) edge (r1)
 (b111) edge (b11)
 (b112) edge (b11)
 (b121) edge (b12)
 (b122) edge (b12)
 (b211) edge (b21)
 (b212) edge (b21)
 (b221) edge (b22)
 (b222) edge (b22)
 ;
 \end{tikzpicture}\]
 
\caption{\label{fig:treegroup} An example of cyclic group labeling of a binary tree $\mathrm{T}_{_{3}}$ by $\Gamma = \langle a \ | \ a^{16}=1\rangle.$}	
\end{figure}
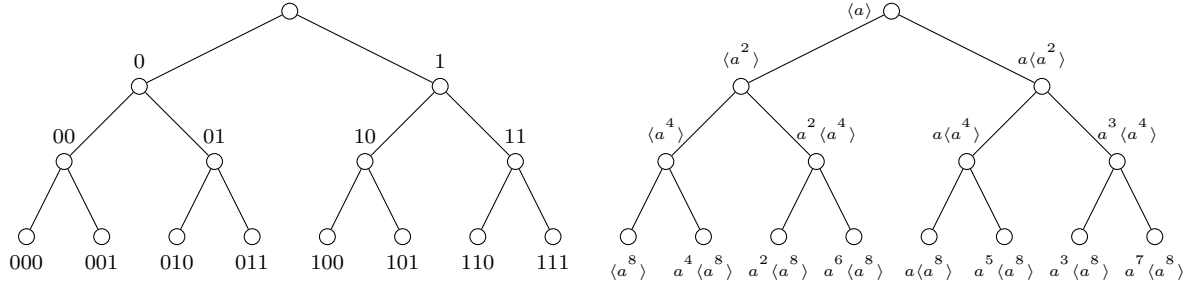

\begin{exm}{\label{exm:eightleaves}{\bf Highly symmetric graphs using $\overline{\mathrm{T}}_{_{3,h}}$}\\
	
	Let $t=2^{^{h+1}}$ such that $n=2t+1$ is a prime number. Consider the multiplicative group $\Gamma = \mathbb{Z}^*_{_{n}}$ of the field $\mathbb{Z}_{_{n}}$ and let $a$ be its generator for which $a^{2t}=1$.
	Also, let $\gamma$ be the cyclic group labeling of a complete binary tree of height $h+1$ by $\Gamma$ and identify this labeling on 
	${\overline{\mathrm{T}}}_{_{3,h}}$ by just forgetting the root.
	
	It is easy to verify that for this labeling the coset which is assigned to a vertex $s$ is exactly the union of the cosets of size two which are assigned to the leaves of the subtree below $s$.  
	Also, it is straight forward to see that multiplying all labels of the labeled tree ${\overline{\mathrm{T}}}_{_{3,h}}$ with the labeling $\gamma$ 
	by an arbitrary element $g \in \Gamma$  gives rise to an automorphism of this labeled tree.
		
	For the leaf $i$, $0\leq i\leq t-1$, let $\gamma(i)^*$ be the element of the coset $\gamma(i)=\{\gamma(i)^*,-\gamma(i)^*\}$ 
	which is greater than or equal to $1$ and less than or equal to $t$, and define $k_{_i} \isdef \gamma(i)^*$. Then, consider the 	
	cyclic $t$-decomposition of $\mathrm{K}_{_{n}}$ as 
	$${\cal K}_{_{n}} \isdef (\mathrm{C}_{_{n}}(k_{_0}),\mathrm{C}_{_{n}}(k_{_1}),\cdots,\mathrm{C}_{_{n}}(k_{_{t-1}})),$$ for which  $$\{1,2,\cdots,t\}=\{k_{_{0}},k_{_{1}},\cdots,k_{_{t-1}}\}.$$
		
	Then we may construct the graph 
	${\cal K}_{_{n}}\boxtimes{\overline{\cal H}}$ in which
    $$\overline{{\cal H}}\isdef(\overline{\mathrm{T}}_{_{3,h,0}},\overline{\mathrm{T}}_{_{3,h,1}},\ldots,\overline{\mathrm{T}}_{_{3,h,t-1}})$$
    and we have assigned the cylinder ${\overline{\mathrm{T}}}_{_{3,h,i}}$ to  the edge set of the subgraph $\mathrm{C}_{_{n}}(k_{_i})$ of $K_{_{n}}$.
    
    Now,  by Proposition~\ref{pro:treecyl} and the fact that multiplication of the labels of the tree by a fixed element of $\Gamma$ is an automorphism of the labeled tree itself, we may conclude that all
    term $\phi({\overline{\mathrm{ T}}}_{_{h,3}}+{\overline \Theta}_{_{j}},x)$ are equal for $1 \leq j < n$, and consequently,
    $$\phi({\cal K}_{_{n}}\boxtimes{\overline{\cal H}},x)=\phi({\overline{\mathrm{
    			T}}}_{_{h,3}}+{\overline \Theta}_{_{1}},x)^{^{2^{^{h+2}}}}\phi({\overline{\mathrm{ T}}}_{_{h,3}}+{\overline \Theta}_{_{n}},x).$$
			Note that ${\overline \Theta}_{_{n}}$ is a diagonal matrix whose diagonal entries are zero for any vertex which is not a leaf and is equal to $2$ for the $i$th leaf.
			    
	As an interesting special case, let $h=2$ and we have the $3$-regular graph  ${\cal K}_{_{17}}\boxtimes{\overline{\cal H}}$ on 238 vertices with the 
	characteristic polynomial
	$$(x-3)(x-2)^{^4}(x^2-2x-2)^{^2}(x^2-2)^{^{17}}(x^{^3}-x^{^2}-6x+2)\times$$
	$$ (x^{^{12}}+x^{^{11}}-18x^{^{10}}-16x^{^{9}}+116x^{^8}+88x^{^7}-319x^{^6}-189x^{^5}+345x^{^4}+116x^{^3}-116x^{^2}-12x+7)^{^{16}}$$
	which has an automorphism group of order $262$.

}\end{exm}

\begin{exm}{\label{exm:sixleaves}{\bf Highly symmetric graphs using ${\mathrm{T}}^{^{\bullet}}_{_{3,h}}$ }\\
	
		Let $t=3 \times 2^{^{h-1}}$ such that $n=2t+1$ is a prime number. Consider the multiplicative group $\Gamma = \mathbb{Z}^*_{_{n}}$ of the field $\mathbb{Z}_{_{n}}$ and let $a$ and $b$ be its generators for which $b^3=1$ and $a^{2t/3}=1$.
		
		Label the root of ${\mathrm{T}}^{^{\bullet}}_{_{3,h}}$
		by $\Gamma$ itself and label each of its siblings as $s$ with $|s|=1$
		by the coset $b^s \langle a \rangle$.
		Also, let $\gamma$ be the cyclic group labeling of a complete binary tree of height $h-1$ under the vertex $0$ 	
		by the cyclic group $\langle a \rangle$ and 
		for the subtrees below the vertices $1$ and $2$ apply translation (i.e. multiplying the cosets) by $b$ and $b^2$, respectively, of the labeling $\gamma$ of the subtree under $0$.
		
		Again, for the leaf $i$,  let $\gamma(i)^*$ be the element of the coset $\gamma(i)=\{\gamma(i)^*,-\gamma(i)^*\}$ 
		which is greater than or equal to $1$ and less than or equal to $t$, and define $k_{_i} \isdef \gamma(i)^*$, and consider the 	
		cyclic $t$-decomposition of $\mathrm{K}_{_{n}}$ as 
		$${\cal K}_{_{n}} \isdef (\mathrm{C}_{_{n}}(k_{_0}),\mathrm{C}_{_{n}}(k_{_1}),\cdots,\mathrm{C}_{_{n}}(k_{_{t-1}}))$$ for which  $$\{1,2,\cdots,t\}=\{k_{_{0}},k_{_{1}},\cdots,k_{_{t-1}}\}.$$
		
	   Now, by Proposition~\ref{pro:treecyl} we have 
	   $$\phi({\cal K}_{_{n}}\boxtimes{\cal H}^{^{\bullet}},x)=
	   \phi(T^{^{\bullet}}_{_{h,3}}+\Theta_{_{n}}^{^{\bullet}},x)\phi(T^{^{\bullet}}_{_{h,3}}+\Theta_{_{1}}^{^{\bullet}},x)^{^{3\times2^h}}.$$
	  Note that $\Theta_{_{n}}^{^{\bullet}}$ is a diagonal matrix whose diagonal entries are zero for any vertex which is not a leaf and is equal to $2$ for the $i$th leaf.

	   In particular, for $h=2$ and $n=13$ we have a $3$-regular 
	   Ramanujan graph on $130$ vertices with the following characteristic polynomial
	   $$\phi({\cal K}_{_{13}}\boxtimes{\cal H}^{^{\bullet}},x)=(x-3)(x-1)(x+2)(x-2)^{^3} (x^{^2}-2x-2)^{^2}$$
	  $$\times (x^{^{10}} +
	  x^{^9}-14x^{^8 }- 12x^{^7}+65x^{^6}+45x^{^5}-115x^{^4}-55x^{^3}
	  +69x^{^2}+12x-10)^{^{12}},$$
	   with an automorphism group of size $156$.
	  
}\end{exm}

\section{Concluding remarks}\label{sec:concluding}

In this article we considered the spectra of bsymmetric cylindrical constructs when the 
base-graph has a commutative decomposition. This in our opinion sets another interesting setup where graph decompositions play a crucial role one the one hand, and presents some 
interesting examples where one may have some observations about the behaviour of the spectrum with a control on the structure of the graph itself. We believe that our results show that the uniform term in Proposition~\ref{pro:treecyl} for $j=n$ described 
in detail in Proposition~\ref{pro:uniformterm} is a consequence of commutativity of the decomposition since this condition gives rise to the fact that the all-one vector should be the eigenvector of all subgraphs of the corresponding commutative decomposition.
This becomes more important when one may verify experimentally that this uniform term is the basic source of eigenvalues outside of the Ramanujan interval, while the number of such eigenvalues for the rest of the terms is relatively very small.

In other words, we believe that our results can be considered as a piece of evidence for the fact that there exists Ramanujan graphs among random sparsifications of the complete graphs by tree-cylinders (e.g. see \cite{mss}). This, in our opinion, serves as a strong source of motivation for the study of such construction, where in particular, random $\pi$-lifts of complete graphs is the specific case that we believe deserves more careful investigation (see \cite{dama}).

From another point of view, our constructions and computations introduce a very specific and important recursive amalgam of polynomials that we believe must be studied in further research for its very interesting properties, since from this more abstract point of view, the recursive procedure presents a way of mixing a couple of polynomials with a control on the spread of the roots of the resultant.
It can be verified that this procedure not only works for the diagonal perturbation of trees but also it goes through for the circulant perturbation of trees too, which may also lead to new constructions.

It worths mentioning that spectral properties of the class of highly symmetric graphs introduced in the last section is closely related to number theoretic properties of
the corresponding parameters and this class of graphs also deserve more investigations 
in future research.

\end{document}